\documentclass[11pt,letterpaper]{amsart}
\usepackage{amssymb,amsmath,latexsym,amscd,amsfonts}
\usepackage{mathrsfs}
\usepackage{mathtools}
\usepackage{hyperref}
\usepackage[svgnames]{xcolor}
\hypersetup{colorlinks,breaklinks,
            linkcolor=DarkBlue,urlcolor=DarkBlue,
            anchorcolor=DarkBlue,citecolor=DarkBlue}
\usepackage{xypic}
\usepackage{psfrag}
\usepackage{graphicx}
\usepackage[verbose,letterpaper,tmargin=.7in,bmargin=.7in,lmargin=.7in,rmargin=.7in]
{geometry}
\usepackage[titletoc,toc,title]{appendix}
\DeclareMathOperator*{\dprime}{\prime \prime}
\newtheorem{theorem}{Theorem}[section]

\newtheorem{proposition}[theorem]{Proposition}
\newtheorem{corollary}[theorem]{Corollary}
\newtheorem{lemma}[theorem]{Lemma}

\newtheorem{fact}[theorem]{Fact}
\theoremstyle{corollry}

\theoremstyle{definition}
\newtheorem{definition}[theorem]{Definition}

\newtheorem{notation}[theorem]{Notation}

\newtheorem{remark*}[theorem]{}
\theoremstyle{remark}
\newtheorem{remark}[theorem]{Remark}

\newcommand{\C}{\mathbb C}

\def\ss{\subset}
\def\la{\langle}
\def\ra{\rangle}

\newcommand{\p}{p}
\allowdisplaybreaks
\begin{document}
\title{ Intermediate Planar Algebra revisited}
\author{Keshab Chandra Bakshi}
\address{Institute of Mathematical Sciences, Taramani, Chennai, India-600113 and Homi Bhabha National
Institute, Training School Complex, Anushakti Nagar, Mumbai, India-400094}
\email{keshabcb@imsc.res.in}

\begin{abstract}
 In this paper, we explicitly work out the subfactor planar algebra $P^{(N \subset Q)}$ for an intermediate subfactor
$N \subset Q \subset M$ of an irreducible subfactor $N \subset M$ of finite
index. We do this in terms of the subfactor planar algebra $P^{(N \subset M)}$
by showing that if $T$ is any planar tangle, the associated operator $Z^{(N \subset Q)}_T$ can be read off from $Z^{(N \subset M)}_T$ by a formula involving the
so-called {\em biprojection} corresponding to the intermediate subfactor $N \subset Q \subset M$ and a scalar $\alpha(T)$ carefully chosen so as to ensure that the formula defining $Z^{(N \subset Q)}_T$ is multiplicative with respect to composition of tangles.
Also, the planar algebra of $Q \ss M$
can be obtained by applying these results to $M \ss M_1$. We also apply our result to the example of a semi-direct product subgroup-subfactor.
\end{abstract}
\maketitle
\section{Introduction}
In this paper, we consider the situation in which an irreducible subfactor $N \ss M$ has an
intermediate subfactor $Q$; and work out a reformulation of the proof of the fact that the planar algebra $P^{(N\subseteq Q)}$
may be derived from $P^{(N\subseteq M)}$ (see \cite{BhaLa} and \cite{La2})
by requiring that the action of a  planar tangle $T$ is given by Equation (\ref{alpheq}) below. 
It is well-known from earlier work of Bisch \cite{Bi1} - and reformulated in \cite{BiJo2} and \cite{La1} in the planar algebraic terms that we will actually use here - that such intermediate subfactors are in bijective correspondence with so-called {\em biprojections} say $q \in P^{(N\subset M)}_2$. We wish here to
describe (in Theorem~\ref{main1}) the planar algebra of $N \ss Q$ in terms of the planar algebra of $N \ss M$, while the planar algebra of $Q \ss M$
can be obtained by applying these results to $M \ss M_1$. The biprojection $q$ corresponding to the intermediate
subfactor $Q$ gives rise naturally to a mapping $F = \{F_m\}$ from tangles of any colour (say $m$) to partially labelled
tangles of the same colour (see Definition \ref{defF}), and a scalar-valued
function $\alpha$ defined on the collection of all tangles (see Definition \ref{defalpha}), such that $P^{(N \ss Q)}$ may be identified with a planar algebra, call it $P^\prime$, with $P'_n = range (Z_{F(I^n_n)}^{(N \ss M)})$,where $I^n_n$ is the identity tangle of colour $n$ and
the multilinear map $Z^{(N \ss Q)}$ associated to a tangle $T^{k_0}_{k_1,\cdots, k_b}$ is given by
\begin{equation} \label{alpheq}
Z^{(N \ss Q)}_T = \alpha(T)Z^{(N \ss M)}_{F(T)},
 \end{equation}
where both sides are thought of as acting on $\otimes_{i=1}^b P'_{k_i}$.
The slightly involved proof of the above assertion takes some work - see Theorem \ref{main} and Theorem \ref{main1}.

\smallskip
The difference in proofs here and in \cite{BhaLa} stems from the two ways that a planar algebra $P$ can be described, respectively, (i) as in \cite{KodSun} (where one says what the underlying spaces $P_n$ are, and explicitly describes the multilinear operator $Z^P_T$  associated to a planar tangle $T$, and  then verifying that these tangle maps satisfy the necessary compatibility conditions, as in Theorem \ref{main1}), and (ii) by specifying a non-degenerate scalar-valued partition function  $Z$ on $0$-tangles (labelled by $S = \coprod_{k\in ~Col} S_k$) which is invariant under planar isotopy and multiplicative on connected components (as in \cite{Jo2}). Thus, one may say that a `bonus' in our approach is that we know how any planar tangle acts on a vector in its domain.
\smallskip

With our formulation of intermediate planar algebra, as an application, in Section \ref{CPE} we have recovered the result of \cite{LaSu} which establishes an one-one correspondence between a planar
subalgebra $P^{\Theta}$ of the `group planar algebra' which is naturally associated with a group $\Theta$ of automorphisms of the given group $G$ and the planar agebra
corresponding to the `subgroup-subfactor' associated with the inclusion $\Theta \subset (G\rtimes \Theta)$. It seems that there is a slight inaccuracy with the constant in the defining
isomorphism ${\beta}_k$ of \cite{LaSu},while the corrected constant may be found in  Definition \ref{defi}.
\smallskip

In an appendix, we have described the tower of iterated basic construction of  $N\subseteq Q$ in terms of the the corresponding tower of $N \subseteq M$. This is
the crucial step in obtaining standard invariant of $N\subset Q$ in terms of $N\subset M$.
The Jones' tower of $N\subset Q$ in terms of $N\subset M$ was described in \cite{BhaLa}. Here we give another proof using yet another characterization of the basic construction, in terms of Pimsner Popa bases. It should be mentioned that D. Bisch gave a partial description of standard invariant of $N\subset Q$ in \cite{Bi2} giving the standard invariant of the inclusion $N\subset Q_1$, where $Q_1$ is the first step basic construction for $N\subset Q$.
\smallskip

After posting this paper on the arXiv, the author has been requested by D. Bisch to mention that there would be a section on intermediate planar algebra involving the same tangles as in Definition \ref{defF} in unpublished joint work with V. Jones, for which a preprint is forthcoming.

\section{Notation and some basic facts}

In this paper, all factors will be of type ${\rm II}_1$, and all subfactors $N\ss M$ will be of finite index $[M:N]$.  By $tr_M$ we will mean the unique normal faithful trace defined on $M$.  $E^M_N$ will denote the trace preserving conditional expectation from $M$ onto $N$; we shall often omit $M$ and write $E_N$ when doing so is unambiguous.\par
 
Following Bisch, we denote the `Jones towers' built from the basic construction for $N \ss Q \ss M$  as:
\[ N \ss Q  \ss M \ss P_1 \ss M_1 \ss P_2 \ss M_2 \ss P_3 \ss \dots \ss M_{2n+1}~. \]
We write $e_{\epsilon,i}, \epsilon \in \{0,1\}, i \geq 1$ for the  projections:
\[ e_{0,i}: L^2(M_{i-1}) \rightarrow L^2(P_{i-1}) \mbox{ \ \ \ \ \ and \ \ \ \ \ } e_{1,i}: L^2(M_{i-1}) \rightarrow L^2(M_{i-2})\]
so that $P_i = \la M_{i-1}, e_{0,i} \ra$ and $M_i=\la M_{i-1}, e_{1,
i}\ra$ (here we set $M_0=M$, $M_{-1}=N$, $P_0=Q$).

The description of the algebras generated by $e_{0,i}$ and $e_{1,i}$
are given in \cite{BiJo1}. We will use the following relations
appearing in \cite{BiJo1}.  In what follows, as usual,
$[a,b]=0$ means that $a$ and $b$ commute, and $[a,B]=0$ means that $a$
commutes with all elements of the set $B$.
\begin{fact} \label{f:erel}
The following relations hold:
\begin{enumerate}
\item $e_{0,i}e_{1,i}=e_{1,i}$,
\item $[e_{a,i},e_{b,j}]=0$ for $|i-j|\geq 2$,
\item $[e_{0,i},e_{0, i\pm 1}]=0$,
\item $[e_{0,i},\  P_{i-1}]=[e_{1,i},\  M_{i-2}]=0$,
\item for $i$ even, $e_{0,i}e_{1,i \pm1}e_{0,i} =[Q:N]^{-1} e_{0,i}e_{0,i\pm1}$, and  $e_{1,i}e_{0,i \pm1}e_{1,i} = [M:Q]^{-1} e_{1,i} $,   
\item for $i$ odd,  $e_{0,i}e_{1,i \pm1}e_{0,i} =[M:Q]^{-1} e_{0,i}e_{0,i\pm1}$, and $e_{1,i}e_{0,i \pm1}e_{1,i} = [Q:N]^{-1} e_{1,i} $,
\item $e_{1,i}e_{1,i \pm1}e_{1,i} =[M:N]^{-1} e_{1,i}  $.
\end{enumerate}
\end{fact}
\begin{proof}
 For a proof look at \cite{BiJo1} Proposition $5.1$.
 \end{proof}
 V. Jones introduced his theory of planar algebras in
\cite{Jo2}. A summary of planar algebra terminology is
given in \cite{La1} and also a crash course on planar algebras is given in \cite {KodSun}. We will mainly follow the notation for planar algebras from \cite{KodSun}(section 2). Thus, we write $P_k$ for the $k$-box space $N^\prime \cap M_{k-1}$, $\delta = [M:N]^{-1/2}$ and write $Z_T$ for the multilinear operator corresponding to a planar tangle $T$.
For each disc, one of its boundary arcs is distinguished and marked with a $*$ placed near it (whereas in \cite{KodSun},$*$ was marked to a distinguished point).
As is usual, we will normally draw the discs as boxes with their $*$ arcs unmarked and assumed to contain their north-west corner (and in exceptional cases when it has been necessary to use a `2-click rotation, as in the following figure, for instance, the $*$-interval will be explicitly marked); typically, when a 2-box has a $q$ in it, the $*$-arc has to be in a white arc, and for biprojections, it is immaterial which white arc has the $*$, and we may omit indicating the $*$. Similarly,  we shall  sometimes omit drawing the external disc. (If from the context the shading is clear we will omit that also.)
If $r\in P_2$ sometimes we also write 
$\begin{minipage}{.1\textwidth}
  \centering
  \includegraphics[scale= .4]{su1.eps}
 \end{minipage}
$ \hspace{2mm} for  \hspace{2mm}$
\begin{minipage}{.1\textwidth}
  \centering
  \includegraphics[scale= .4]{su2.eps}
 \end{minipage}$.
\par
It is well-known from \cite{Bi1}, \cite{La1} and \cite{BiJo2}  that in case $N^\prime \cap M = \C$, there is a bijective correspondence between biprojections $q$ (corresponding to the Jones projection of $L^2(M)$ onto $L^2(Q)$) and the intermediate subfactor $Q$, where  $N \subset Q \subset M$. More precisely, we have the following (reformulation of) Theorem 3.2 of \cite{Bi1}.
\begin{theorem}\label{Bisch}\cite{Bi1} \cite{La1} \cite{BiJo2}
 Let $N \ss M$  be an extremal ${\rm II}_1$ subfactor.
Let $P^{(N\ss M)}$ be the planar algebra of $N \ss M$, and let
$\Phi_{N\ss M}$ be the presenting map of $P^{(N \ss M)}$ on itself, i.e. $\Phi_{N\ss M}: \mathcal{P}
(L) \rightarrow P^{(N\ss M)}$ with $L = \coprod P^{(N \ss M)}_{k}$.  Suppose there exists an intermediate subfactor $Q$, $N \ss Q \ss M$.  
If we let $\begin{minipage}{.05\textwidth}
            \centering
            \includegraphics[scale=.35]{q.eps}
           \end{minipage}
$ denote the biprojection  corresponding to $Q$, we have \vspace{2mm}\\
\begin{tabular}{ll}
a) \ \ $\begin{array}[c]{l} \includegraphics[scale=.45]{bisa.eps} \end{array}$ &
b) \ \ $\begin{array}[c]{l}  \includegraphics[scale=.45]{bisb.eps} \end{array}$ \vspace{2mm}\\
c) \ \ $\begin{array}[c]{l}  \includegraphics[scale=.45]{bisc.eps} \end{array}$ &
d) \ \  $\begin{array}[c]{l} \includegraphics[scale=.4]{bisd.eps} \end{array}$
\end{tabular}

 with $c = [M:N]^{1/2}[M:Q ]^{-1}$. Furthermore, in the case $N'\cap M=\C$, the converse is also true. 
 Namely  a 2-box  $\begin{minipage}{.05\textwidth}
                    \centering
                    \includegraphics[scale=.5]{q.eps}
                   \end{minipage}
$ satisfying a)-d) above implies the 
 existence of an intermediate subfactor $Q$, $N \ss Q \ss M$ corresponding to
  $\begin{minipage}{.05\textwidth}
    \centering
    \includegraphics[scale=.5]{q.eps}
   \end{minipage}
$.
\end{theorem}

\begin{corollary}\label{cor:exchange}
The following exchange relation holds:
\begin{figure}[h!]
  \begin{minipage}{1\textwidth}
    \begin{equation*}
    {\begin{minipage}{.4\textwidth}
    \centering
    \includegraphics[scale=.4]{flower.eps}
    \end{minipage}}
    = 
    {\begin{minipage}{.4\textwidth}
    \centering
    \includegraphics[scale=.4]{flower2.eps}
    \end{minipage}}
    \end{equation*}
  \end{minipage}
  \end{figure}
\end{corollary}

Denote by $N \ss Q 
\ss Q_1 \ss Q_2 \ss \dots$ the Jones tower of $N \ss Q $. 
\begin{definition}\label{defF}
 Denote the following tangle by $E_n$:
 $$ \includegraphics [scale=.5] {Def.eps}$$
according as $n$ is even or odd respectively. We shall use these to define a map $T \mapsto F(T)$ from the class of $k$-tangles to the class of {\em partially labelled} $k$-tangles with $(k+1)$ internal discs all but the last of which are 2-boxes labelled with a $q$, with the tangle $T$ inserted in the last disc of colour $k$. Thus, $F(T) = E_k\circ_{ (D_1,D_2,\cdots, D_k, D_{k+1})}(q,q,\cdots,q,T)$.

If it is clear from the context then we write $E$ instead of $E_n$. \par Define functions  $F_n :P_n \mapsto P_n$ by $F_n(x)=Z_{E_n}(q \otimes q \otimes \cdots \otimes q \otimes x)$
for $x \in P_n$. We often write $F(x)$ instead of $F_n(x)$ if there is no confusion.
\end{definition}
Following \cite{BhaLa} define the natural inclusion map 
\[
i:  F_n(P_n) \rightarrow F_{n+1}(P_{n+1}) 
\]
given by
\[ 
t \mapsto F_{n+1}(t).  
\]
We denote this inclusion by $\ss_i$.  Our starting point is the following result from \cite{BhaLa}:
\begin{theorem}\label{planarstinv}
The lattice of algebras:
\[
\begin{array}{ccccccccccc}
F_{0}(P_0) & \ss_i & F_{1}(P_1) & \ss_i &F_{2}(P_2)  &\ss_i &\dots&
\ss_i& F_{n}(P_{n}) & \ss_i & \dots \\
&&\cup && \cup && &&\cup && \\
&& F_{1}(P_{1,1}) & \ss_i & F_{2}(P_{1,2}) & \ss_i & \dots & \ss_i &
F_{n}(P_{1,n})& \ss_i & \dots 
\end{array}\]
is isomorphic to
the standard invariant of $N\ss Q $:
\[
\begin{array}{ccccccccccc}
N'\cap N & \ss & N' \cap Q & \ss & N' \cap Q_1 & \ss&
\dots& \ss& N' \cap Q_{n-1} &\ss & \dots \\
&&\cup && \cup &&&& \cup && \\
&&Q ' \cap Q &\ss & Q ' \cap Q_1 &\ss& \dots & \ss & Q '
\cap Q_{n-1} & \ss & \dots
\end{array}\]
The Jones projections are 
$$ F_{2n+1}(P_{2n+1}) \ni e^{Q }_{2n} = [M:Q ]^{1/2}[Q :N]^{-1/2}
{\begin{minipage}{.4\textwidth}
    \centering
    \includegraphics[scale=.5]{jopo1.eps}
    \end{minipage}}
     $$
and
$$ F_{2n+2}(P_{2n+2}) \ni e^{Q }_{2n+1}= [M:N]^{-1/2}
     {\begin{minipage}{.4\textwidth}
    \centering
    \includegraphics[scale=.5]{jopo2.eps}
    \end{minipage}} 
    $$
for $n\geq 1$.
The trace on $F_{2n}(P_{2n})$ and $F_{2n+1}(P_{2n+1}))$ is given by
$tr_{N \ss Q }(x) = [M:Q ]^{n} tr_{N \ss M}(x) $.
\end{theorem}

\section{The Intermediate Planar Algebra}
This section is devoted to our reformulation of the proof of the fact that the planar algebra $P^{(N\subseteq Q)}$
may be derived from $P^{(N\subseteq M)}$ (see \cite{BhaLa})
by requiring that the action of a  planar tangle $T$ is given by Equation (\ref{alpheq}) of the Introduction.
\label{sec:planar-algebra-an}

\begin {definition}\label{defalpha}
  Let $ T $ be a $k$-tangle with $ b\geq 1 $ internal discs ${D_1,\dots D_b}$ of colours ${k_1,\dots k_b}$. Then define $\alpha(T) = [M:Q]^{\frac{1}{2}c(T)}$, where
\[c(T)=(\lceil k_0/2 \rceil+\lfloor k_1/2 \rfloor+\dots +\lfloor k_b/2 \rfloor)-l(T) \]
with $l(T)$ being the number of closed loops after capping the black intervals of the external disc  of $T$ and cupping the black intervals of all internal discs of $T$. 
\end {definition}

\begin{proposition}\label{alphas}
If $T = T^{k_0}_{k_1,\cdots,k_b}$ and $\tilde{T} = \tilde{T}^{\tilde{k}_0}_{\tilde{k}_1,\cdots,\tilde{k}_{\tilde{b}}}$ are tangles with discs of indicated colours such that $\tilde{k}_0=k_i$ for some $1 \leq i \leq b$, then
\[\frac{\alpha(T)\alpha(\tilde{T})}{\alpha(T\circ_i\tilde{T})} =  [M:Q]^{\frac{1}{2}(\tilde{k_0} - l(T) - l(\tilde{T}) + l(T\circ_i\tilde{T}))}.\]
\end{proposition}

\begin{proof} This is simple arithmetic:
\begin{eqnarray*}
c(T) &=&  (\lceil k_0/2 \rceil+\lfloor k_1/2 \rfloor+\dots +\lfloor k_b/2 \rfloor)-l(T)\\
c(\tilde{T}) &=& (\lceil \tilde{k}_0/2 \rceil+\lfloor \tilde{k}_1/2 \rfloor+\dots +\lfloor \tilde{k}_{\tilde{b}}/2 \rfloor)-l(\tilde{T})\\
c(T\circ_{D_i} \tilde{T}) &=& (\lceil k_0/2 \rceil+\lfloor k_1/2 \rfloor+\dots +\lfloor k_b/2 \rfloor)- \lfloor \tilde{k}_0/2 \rfloor + \lfloor \tilde{k}_1/2 \rfloor+\dots +\lfloor \tilde{k}_{\tilde{b}}/2 \rfloor- l(T\circ_i \tilde{T})
\end{eqnarray*}
Hence, after all the cancellation, we find that
\begin{eqnarray*}
\frac{\alpha(T)\alpha(\tilde{T})}{\alpha(T\circ_i\tilde{T})} &=&
[M:Q]^{1/2\left(c(T) +c(\tilde{T})-c(T\circ_i \tilde{T})\right)}\\
&=& [M:Q]^{1/2\left(\tilde{k}_0 - l(T) - l(\tilde{T}) + l(T\circ_i\tilde{T}) \right)}~,
\end{eqnarray*}
since $\lceil n/2 \rceil + \lfloor n/2 \rfloor = n$ for all integral $n$. 
\end{proof}

We shall show that:

\begin{theorem}\label{main1} If $P^\prime_k = ran(F(I^k_k))$ and $Z^\prime_T = \alpha(T) Z_{F(T)|_{\otimes P^\prime_{k_i(T)}}}$, then $(P^\prime, T \mapsto Z^\prime_T|_{\otimes P^\prime_{k_i(T)}})$ is a subfactor planar algebra which is isomorphic to $P^{(N\ss Q)}$. 
\end{theorem}

The proof of this theorem has two main ingredients: (a) the verification that $P^\prime$ is a planar algebra; and (b) the verification that this is isomorphic to the planar algebra of $N \ss Q$. The proof of (b) is an application of a theorem of Jones (see \cite{KodSun} Theorem 2.1,which is the formulation that we shall use) while the only really non-trivial part of proving (a) is in the verification of compatibility of the partition function to gluing of tangles. In order to verify that the operation of tangles (in $P^\prime$) is compatible with composition of tangles, we will need to verify that
\[\alpha(T\circ \tilde{T}) Z_{F(T \circ \tilde{T})} = \alpha(T)\alpha(\tilde{T}) Z_{F(T)\circ F(\tilde{T})}~,\]
which, in view of Proposition \ref{alphas}, is seen to translate to:
\begin{equation}\label{TPT}
Z_{F(T \circ F(\tilde{T}))} = [M:Q]^{-1/2(\tilde{k}_0 - l(T) -l(\tilde{T_0}) + l(T\circ \tilde{T})} Z_{F(T \circ \tilde{T})}~,
\end{equation}
which is what the next few pages are devoted to. We start on part (b) in the few lines after the proof of ``compatibility under substitution''.

\bigskip \noindent
Thus, we assume that $P$ is an irreducible subfactor planar algebra and $q \in P_2$ be a biprojection. Let $T = T_{k_1,\cdots,k_b}^{k_0}$ and $\tilde{T} = \tilde{T}_{\tilde{k}_1,\cdots,\tilde{k}_{\tilde{b}}}^{\tilde{k}_0}$ be tangles with $k_i = \tilde{k}_0$. By $F(T)$ we will denote the partially labelled tangle obtained from $T$
by `surrounding it with $q$'s'. It has the same number and colours of discs as $T$ does. Let $P^\prime_k$ be the range  of the tangle $F(I_k^k)$.

\begin{theorem}\label{main} The equation
\begin{equation*}
Z_{F(T) \circ_i F(\tilde{T})} = \tau(q)^{\frac{1}{2}({k}_i + l(T \circ_i \tilde{T}) - l(T) - l(\tilde{T}))} Z_{F(T \circ_i \tilde{T})},
\end{equation*}
holds for inputs coming from $P^\prime$.
\end{theorem}

Here, for any tangle $T$, $l(T)$ is the number of loops obtained after black-cupping the internal discs of $T$ and black-capping the external
disc of $T$.

The proof of Theorem \ref{main} proceeds by a series of reductions to easier and easier cases until the result is obvious. There are 4 main steps.

\bigskip\noindent
{\bf Step 1:} Reduction to the case $T$ is a $0_+$-tangle : Let $S$ be the $0_+$-tangle in Figure~\ref{step}
below and $\tilde{S} = \tilde{T}$. We claim that the truth of the equation for
$S$ and $\tilde{S}$ implies it for $T$ and $\tilde{T}$. The new disc of $S$ is the last numbered one.
\begin{figure}[!h]
\begin{center}
\psfrag{hatt}{}
\psfrag{k}{\huge $k_0$}
\psfrag{1}{\huge $T$}
\resizebox{3cm}{!}{\includegraphics{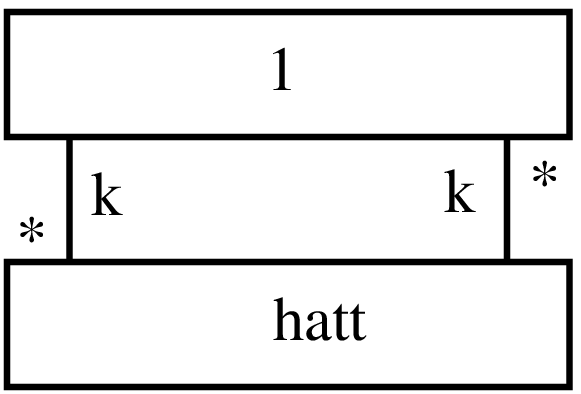}}
\end{center}
\caption{}\label{step}
\end{figure}

Observe that, by definition, $l(S) = l(T), l(\tilde{S}) = l(\tilde{T})$ and $l(S \circ_i \tilde{S}) =
l(T \circ_i \tilde{T})$. To prove Theorem \ref{main}, it suffices to trace both sides against an arbitrary
element $x \in P_{k_0}$ and verify that the results are the same.

Now $\delta^{k_0} \tau(Z_{F(T) \circ_i F(\tilde{T})}(\cdots) x) =  Z_{F(S) \circ_i F(\tilde{S})}(\cdots, F(x))$ and \par
$\delta^{k_0} \tau(Z_{F(T \circ_i \tilde{T})}(\cdots) x) =  Z_{F(S \circ_i \tilde{S})}(\cdots, F(x))$

Also, we are given that
\begin{equation*}
Z_{F(S) \circ_i F(\tilde{S})} = \tau(q)^{\frac{1}{2}({k}_i + l(S \circ_i \tilde{S}) - l(S) - l(\tilde{S}))} Z_{F(S \circ_i \tilde{S})},
\end{equation*}
holds when all inputs come from $P^\prime$. 

Thus the $\delta^{k_0} \tau(Z_{F(T) \circ_i F(\tilde{T})}(\cdots) x) $ above equals
\begin{eqnarray*}
\tau(q)^{\frac{1}{2}({k}_i + l(S \circ_i \tilde{S}) - l(S) - l(\tilde{S}))} Z_{F(S \circ_i \tilde{S})}(\cdots,F(x)) = \\
\delta^{k_0} \tau(q)^{\frac{1}{2}({k}_i + l(T \circ_i \tilde{T}) - l(T) - l(\tilde{T}))} \tau(Z_{F(T \circ_i \tilde{T})}(\cdots)x).&
\end{eqnarray*}
The desired reduction follows.
This reduction having been made, we will henceforth assume that $T$
 is a $0_+$-tangle and therefore the equation that must be seen to hold on $P^\prime$ is:
\begin{equation*}
Z_{T \circ_i F(\tilde{T})} = \tau(q)^{\frac{1}{2}({k}_i + l(T \circ_i \tilde{T}) - l(T) - l(\tilde{T}))} Z_{T \circ_i \tilde{T}}. {\hspace {1in}}  (*)
\end{equation*}
(since $F(T) =T$ for a $0_+$-tangle $T$).

\bigskip\noindent
{\bf Step 2:} Reduction to the case $T$ is of the form in Figure \ref{one}
\begin{figure}[!h]
\begin{center}
\psfrag{hatt}{\huge $\hat{T}$}
\psfrag{k}{\huge $k_i$}
\psfrag{1}{\huge $1$}
\resizebox{3cm}{!}{\includegraphics{reduction2.eps}}
\end{center}
\caption{}\label{one}
\end{figure}
\noindent
where $\hat{T}$ is some ${k}_i$-tangle and $i=1$: This follows from sphericality.

\bigskip\noindent
{\bf Step 3:} Reduction to the case $\tilde{T}$ is Temperley-Lieb: This is handled in two different ways according as ${k}_i$ is
even or odd.\\
Subcase 3.1: Suppose that ${k}_i$ is even.
Let $U = U^{k_1}_{2k_1,k_1}$ and $\tilde{S} = \tilde{S}^{2k_1}$ be the following tangles in Figure \ref{subcase3.1},
\begin{figure}[!h]
\begin{center}
\psfrag{k}{\huge $k_1$}
\psfrag{1}{\huge $1$}
\psfrag{2}{\huge $2$}
\psfrag{u=}{\huge $U=$}
\psfrag{s=}{\huge $\tilde{S}=$}
\resizebox{12cm}{!}{\includegraphics{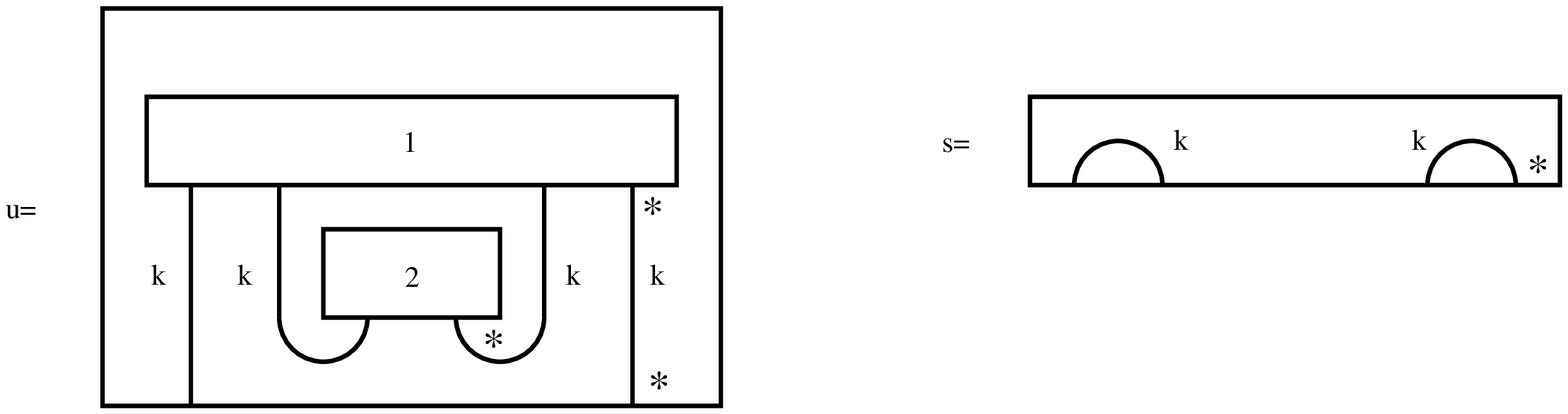}}
\end{center}
\caption{}\label{subcase3.1}
\end{figure}
and $S = T \circ_1 (U \circ_2 \tilde{T})$. It is then clear that $S \circ_1 \tilde{S} = T \circ_1 \tilde{T}$. 

We claim that the validity of Equation (*) holding for the pair $(S,\tilde{S})$ implies its validity for the pair $(T,\tilde{T})$. To see this, assume that
\begin{equation*}
Z_{S \circ_1 F(\tilde{S})} = \tau(q)^{\frac{1}{2}(2{k}_1 + l(S \circ_1 \tilde{S}) - l(S) - l(\tilde{S}))} Z_{S \circ_1 \tilde{S}}. 
\end{equation*}
Now observe that $Z_{S \circ_1 \tilde{S}} = Z_{T \circ_1 \tilde{T}}$ and $Z_{S \circ_1 F(\tilde{S})} = Z_{T \circ_1 F(\tilde{T})}$ since $k_1$ is even and
using that $q^2 = q$ several times. Also, note that $l(\tilde{S}) = k_1$ and $l(S) = l(\tilde{T}) + l(\hat{T}) = l(\tilde{T}) + l(T)$. Substituting all this
in the previous equation and simplifying, we get the desired Equation (*).\\
Subcase 3.2: Suppose that ${k}_1$ is odd. Now let $U = U^{k_1}_{2(k_1+1),k_1}$ and $\tilde{S} = \tilde{S}^{2(k_1+1)}$ be the following tangles in Figure \ref{subcase3.2}
\begin{figure}[!h] 
\begin{center}
\psfrag{k}{\huge $k_1$}
\psfrag{1}{\huge $1$}
\psfrag{2}{\huge $2$}
\psfrag{l}{\huge {$k_1+1$}}
\psfrag{u=}{\huge $U=$}
\psfrag{s=}{\huge $\tilde{S}=$}
\resizebox{12cm}{!}{\includegraphics{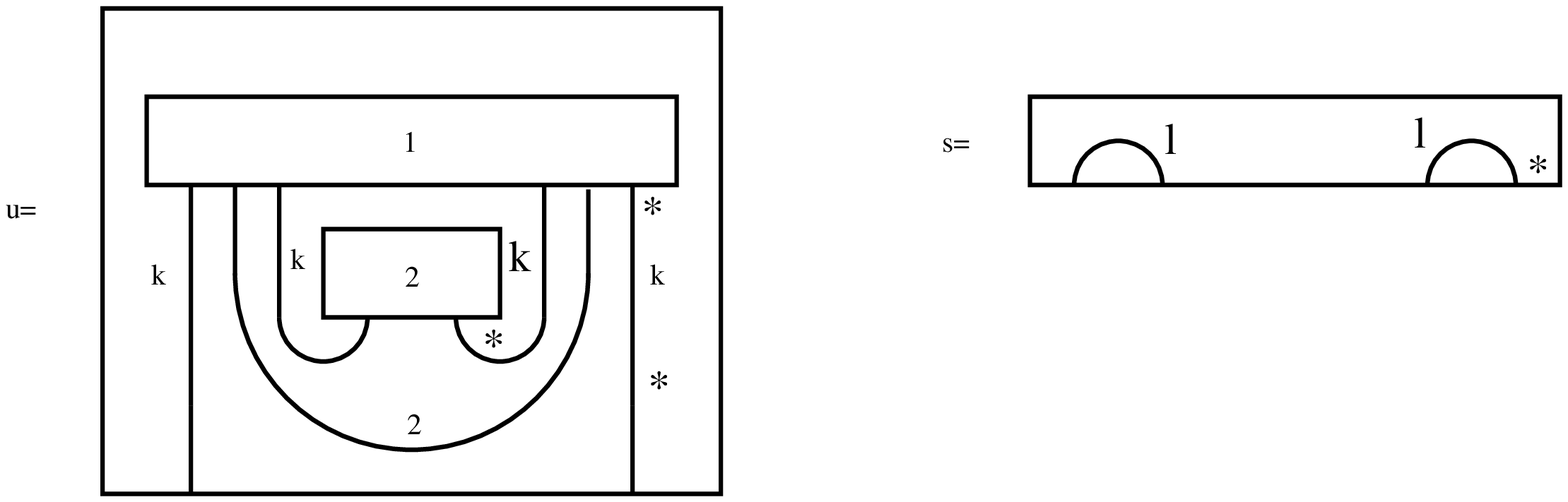}}
\end{center}
\caption{}\label{subcase3.2}
\end{figure}
and let $S = T \circ_1 (U \circ_2 \tilde{T})$. It is then clear that $S \circ_1 \tilde{S}$ differs from $T \circ_1 \tilde{T}$ in having one extra floating loop.

We again claim that the validity of Equation (*) holding for the pair $(S,\tilde{S})$ implies its validity for the pair $(T,\tilde{T})$. To see this, assume that
\begin{equation*}
Z_{S \circ_1 F(\tilde{S})} = \tau(q)^{\frac{1}{2}(2({k}_1+1) + l(S \circ_1 \tilde{S}) - l(S) - l(\tilde{S}))} Z_{S \circ_1 \tilde{S}}. 
\end{equation*}
Now observe that $Z_{S \circ_1 \tilde{S}} = \delta Z_{T \circ_1 \tilde{T}}$. Also note that $l(\tilde{S}) = k_1+1$, $l(S) = l(\tilde{T}) + l(\hat{T}) = l(\tilde{T}) + l(T)$, and $l(S \circ_1 \tilde{S}) = l(T \circ_1 \tilde{T})+1$.

To finish the proof, it suffices to see that $Z_{S \circ_1 F(\tilde{S})} = \delta \tau(q) Z_{T \circ_1 F(\tilde{T})}$. We will first do this in the case
$k_1=5$.
\begin{figure}[!h]
\begin{center}
\psfrag{that}{\Huge $\hat{T}$}
\psfrag{ttilde}{\Huge $\tilde{T}$}
\psfrag{q}{\Huge $q$}
\psfrag{1}{\huge $1$}
\resizebox{9cm}{!}{\includegraphics{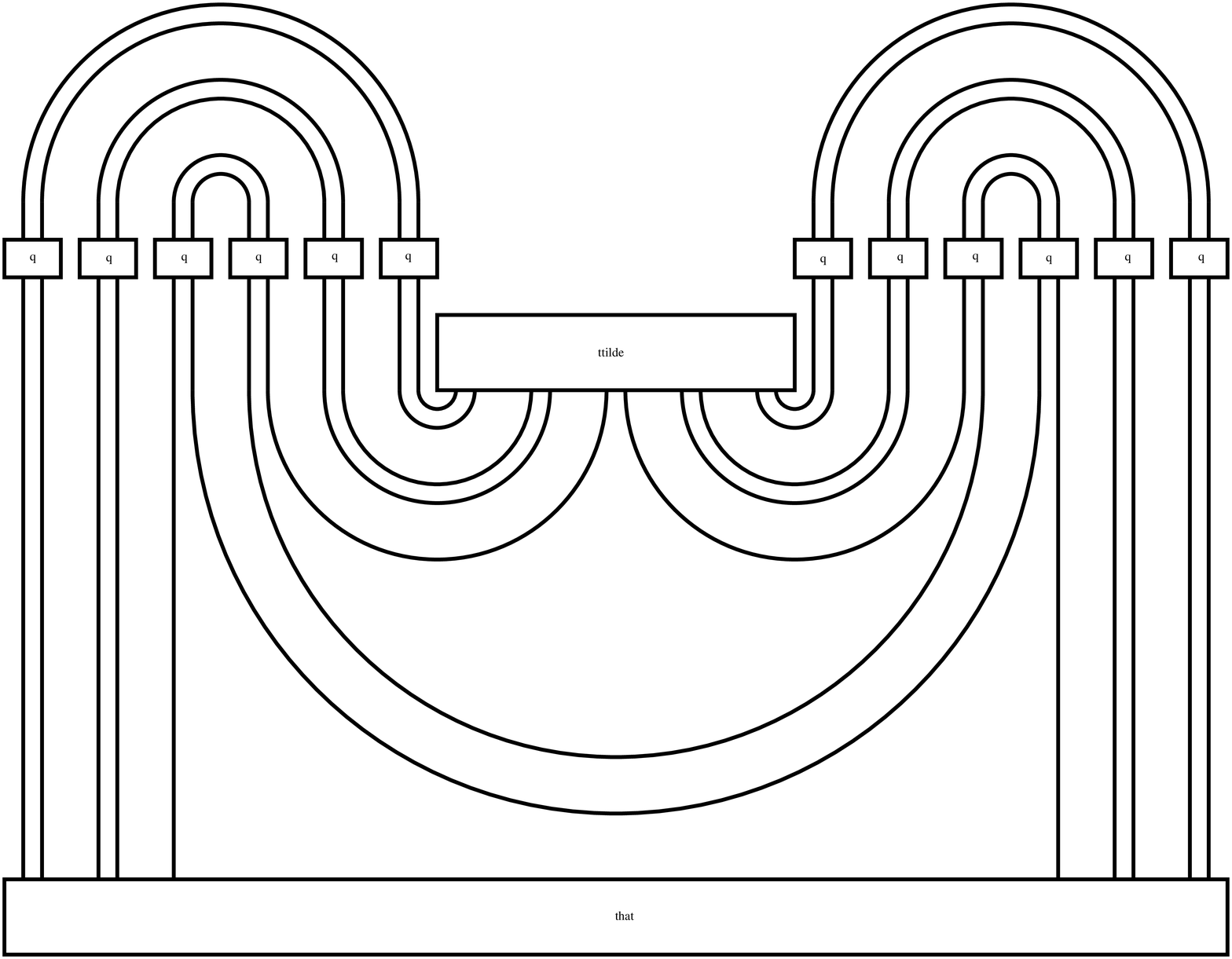}}
\caption{}\label{complicated}
\end{center}
\end{figure}
The tangle $S \circ_1 F(\tilde{S})$ is depicted in Figure \ref{complicated}. With a little bit of manipulation, this reduces to Figure \ref{simplified}.
\begin{figure}[!h]
\begin{center}
\psfrag{that}{\Huge $\hat{T}$}
\psfrag{ttilde}{\Huge $\tilde{T}$}
\psfrag{q}{\Huge $q$}
\psfrag{=}{\Huge $= \delta \tau(q)$}
\resizebox{12cm}{!}{\includegraphics{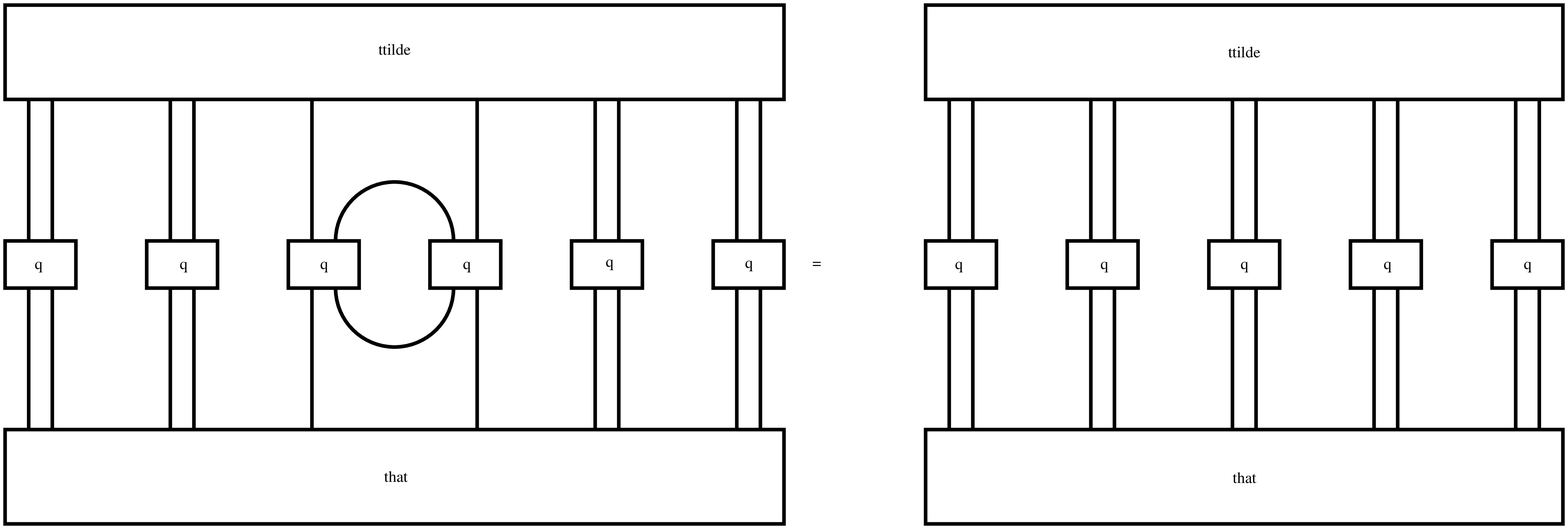}}
\end{center}
\caption{}\label{simplified}
\end{figure}

\noindent
Since the last picture is clearly $T \circ_1 F(\tilde{T})$, we're done. It should be clear that a similar proof works whenever $k_1$ is odd.

\bigskip\noindent
{\bf Step 4:} Resolution of the case $\tilde{T}$ is Temperley-Lieb in three different subcases by induction on ${k}_1$. In each of the subcases,
we will show that the statement for a suitably chosen $S$ and $\tilde{S}$ with $k_0(\tilde{S}) < k_0(\tilde{T})$, implies it for $T$ and $\tilde{T}$.\\
Subcase 4.1: Suppose that in $\tilde{T}$ some $2i-1$ and $2i$ are joined by a string so that $\tilde{T}$ has the form in Figure \ref{two}
\begin{figure}[!h]
\begin{center}
\psfrag{cdots}{\huge ${\cdots}$}
\psfrag{stilde}{\huge $\tilde{S}$}
\psfrag{q}{\Huge $q$}
\psfrag{1}{\large $1$}
\psfrag{2i-1}{\large $2i-1$}
\psfrag{2i}{\large $2i$}
\psfrag{2i+1}{\large $2i+1$}
\psfrag{2i+2}{\large $2i+2$}
\psfrag{2i-2}{\large $2i-2$}
\psfrag{2k_1}{\large $2k_1$}
\psfrag{=}{\Huge $= \delta \tau(q)$}
\resizebox{5cm}{!}{\includegraphics{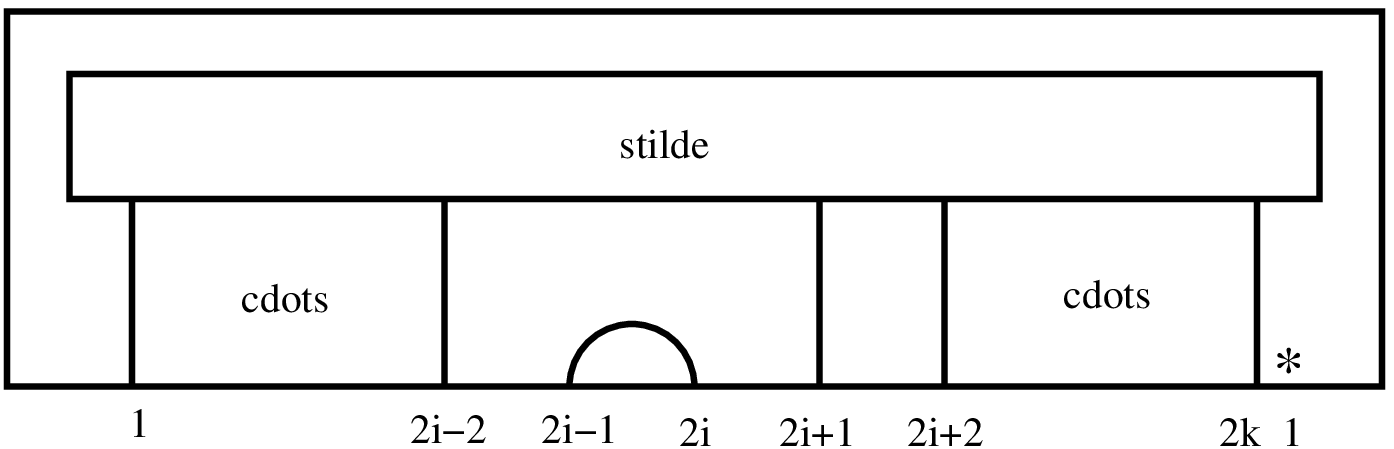}}
\end{center}
\caption{}\label{two}
\end{figure} for some Temperley-Lieb tangle $\tilde{S}$ of colour $k_1-1$. In this case, let $S$ be the  tangle in Figure \ref{three}.
 \begin{figure}[!h]
\begin{center}
\psfrag{cdots}{\huge ${\cdots}$}
\psfrag{ttilde}{\huge $\hat{T}$}
\psfrag{q}{\Huge $q$}
\psfrag{1}{\large $1$}
\psfrag{2i-1}{\large $2i-1$}
\psfrag{2i}{\large $2i$}
\psfrag{2i+1}{\large $2i+1$}
\psfrag{2i+2}{\large $2i+2$}
\psfrag{2i-2}{\large $2i-2$}
\psfrag{2k_1}{\large $2k_1$}
\psfrag{=}{\Huge $= \delta \tau(q)$}
\resizebox{5cm}{!}{\includegraphics{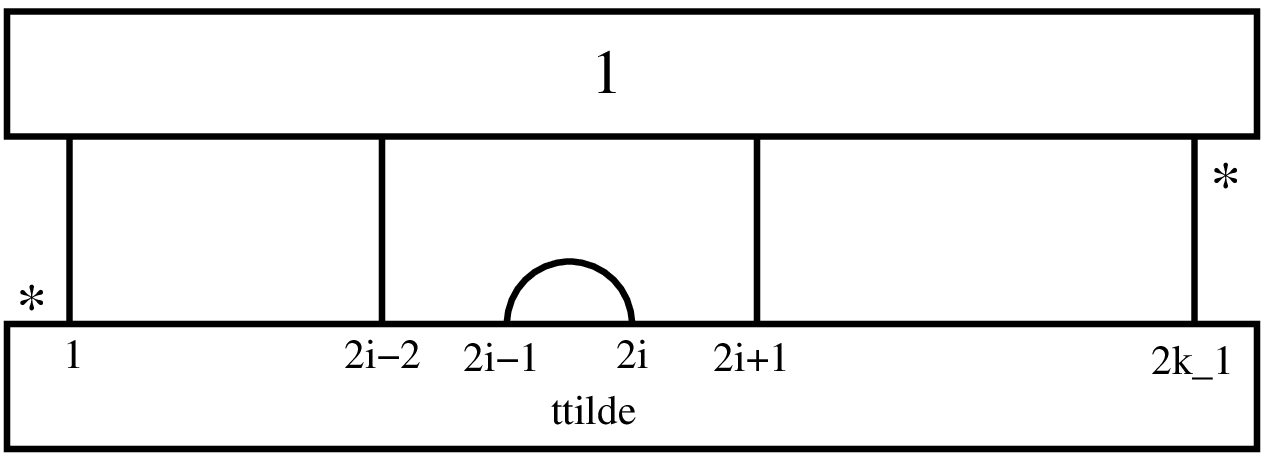}}
\end{center}
\caption{}\label{three}
\end{figure}

Note that $T\circ_1\tilde{T} = S\circ_1\tilde{S}$. It follows that $l(T\circ_1\tilde{T}) = l(S\circ_1\tilde{S})$ and it is easy to see that
$l(S) = l(T)$ ($= l(\hat{T})$) and that $l(\tilde{S}) = l(\tilde{T}) - 1$. To show that the statement for the pair $S,\tilde{S}$ implies that for
the pair $T,\tilde{T}$, it therefore suffices now to see that $T\circ_1F(\tilde{T}) = S\circ_1 F(\tilde{S})$. This follows easily from the 
fact that `$q$ capped on top can be replaced by the identity'.

\noindent
Subcase 4.2: Suppose that in $\tilde{T}$ some $2i$ and $2i+1$ are joined by a string so that $\tilde{T}$ has the form in Figure \ref{subcase4.2}
\begin{figure}[!h]
\begin{center}
\psfrag{cdots}{\huge ${\cdots}$}
\psfrag{stilde}{\huge $\tilde{S}$}
\psfrag{q}{\Huge $q$}
\psfrag{1}{\large $1$}
\psfrag{2i-1}{\large $2i$}
\psfrag{2i}{\large $2i+1$}
\psfrag{2i+1}{\large $2i+2$}
\psfrag{2i+2}{\large $2i+3$}
\psfrag{2i-2}{\large $2i-1$}
\psfrag{2k_1}{\large $2k_1$}
\psfrag{=}{\Huge $= \delta \tau(q)$}
\resizebox{5cm}{!}{\includegraphics{tildet.eps}}
\end{center}
\caption{}\label{subcase4.2}
\end{figure}

\noindent
 for some Temperley-Lieb tangle $\tilde{S}$ of colour $k_1-1$. Note that $l(\tilde{S}) = l(\tilde{T})$. Here there are two further subcases.\\
 Subcase 4.2(a): The black intervals $[2i-1,2i]$ and $[2i+1,2i+2]$ are part of distinct black regions in $\hat{T}$. In this case, let $S$ be the tangle Figure \ref{subcase4.2a}
  \begin{figure}[!h]
\begin{center}
\psfrag{cdots}{\huge ${\cdots}$}
\psfrag{ttilde}{\huge $\hat{T}$}
\psfrag{q}{\Huge $q$}
\psfrag{1}{\large $1$}
\psfrag{2i-1}{\large $2i$}
\psfrag{2i}{\large $2i+1$}
\psfrag{2i+1}{\large $2i+2$}
\psfrag{2i+2}{\large $2i+3$}
\psfrag{2i-2}{\large $2i-1$}
\psfrag{2k_1}{\large $2k_1$}
\psfrag{=}{\Huge $= \delta \tau(q)$}
\resizebox{5cm}{!}{\includegraphics{tangles.eps}}
\end{center}
\caption{}\label{subcase4.2a}
\end{figure}

Here too $T\circ_1\tilde{T} = S\circ_1\tilde{S}$ and it follows that $l(T\circ_1\tilde{T}) = l(S\circ_1\tilde{S})$.
Recall that $l(\tilde{S}) = l(\tilde{T})$. The pictures for computing $l(T)$ and $l(S)$ are shown in Figure \ref{four}.
(The picture for $l(T)$ is above the one for $l(S)$).
 \begin{figure}[!h]
\begin{center}
\psfrag{cdots}{\huge ${\cdots}$}
\psfrag{ttilde}{\huge $BC(\hat{T})$}
\psfrag{q}{\Huge $q$}
\psfrag{1}{\large $1$}
\psfrag{2i-1}{\large $2i-1$}
\psfrag{2i}{\large $2i$}
\psfrag{2i+1}{\large $2i+1$}
\psfrag{2i+2}{\large $2i+2$}
\psfrag{2i-2}{\large $2i-2$}
\psfrag{2k_1}{\large $2k_1$}
\psfrag{=}{\Huge $= \delta \tau(q)$}
\resizebox{8cm}{!}{\includegraphics{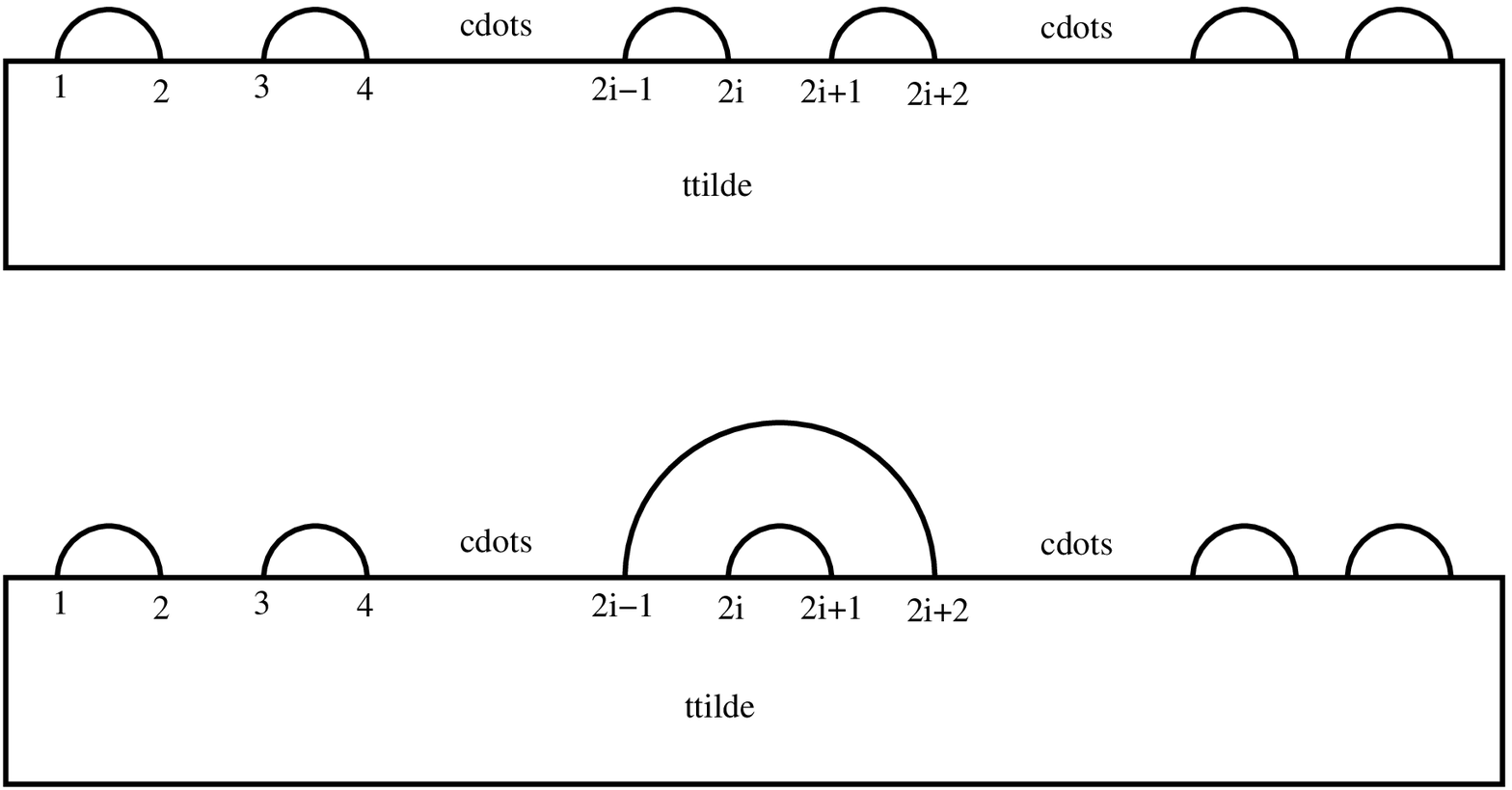}}
\end{center}
\caption{}\label{four}
\end{figure}

Here $BC(\hat{T})$ is the tangle obtained by `black cupping' the insides of all boxes of $\hat{T}$. We need to compare the
number of loops in the top and bottom pictures. Observe that the black regions of $\hat{T}$ and that of $BC(\hat{T})$ are in
natural bijective correspondence and therefore the black intervals $[2i-1,2i]$ and $[2i+1,2i+2]$ are part of distinct black regions in $BC(\hat{T})$.
Thus the loops containing $2i-1$ (and  $2i$) and $2i+1$ (and $2i+2$) are different in the first picture while these two loops are
cut and spliced into a single loop in the second picture. It follows that $l(S) = l(T) -1$.

Now suppose that we know that
\begin{equation*}
Z_{S \circ_1 F(\tilde{S})} = \tau(q)^{\frac{1}{2}({k}_1-1 + l(S \circ_1 \tilde{S}) - l(S) - l(\tilde{S}))} Z_{S \circ_1 \tilde{S}}. 
\end{equation*}
It follows that
\begin{equation*}
Z_{S \circ_1 F(\tilde{S})} = \tau(q)^{\frac{1}{2}({k}_1 + l(T \circ_1 \tilde{T}) - l(T) - l(\tilde{T}))} Z_{T \circ_1 \tilde{T}},
\end{equation*}
and so to complete the proof it suffices to see that $Z_{S \circ_1 F(\tilde{S})} = Z_{T \circ_1 F(\tilde{T})}$.

To see this, first note that the `antipode symmetry' of the $q$ implies that $T \circ_1 F(\tilde{T}) = V \circ_1 F(\hat{T})$ where
$V$ is the tangle in Figure \ref{antipode}.
\begin{figure}[!h]
\begin{center}
\psfrag{hatt}{\huge $1$}
\psfrag{k}{\huge $k_i$}
\psfrag{1}{\huge $\tilde{T}$}
\resizebox{3cm}{!}{\includegraphics{reduction2.eps}}
\end{center}
\caption{}\label{antipode}
\end{figure}

Thus $T \circ_1 F(\tilde{T})$ is given by the picture on the left in Figure \ref{last} which equals the one on the right using properties of $q$.
\begin{figure}[!h]
\begin{center}
\psfrag{that}{\huge $\hat{T}$}
\psfrag{k}{\huge $k_i$}
\psfrag{q}{\huge $q$}
\psfrag{stilde}{\huge $\tilde{S}$}
\resizebox{12cm}{!}{\includegraphics{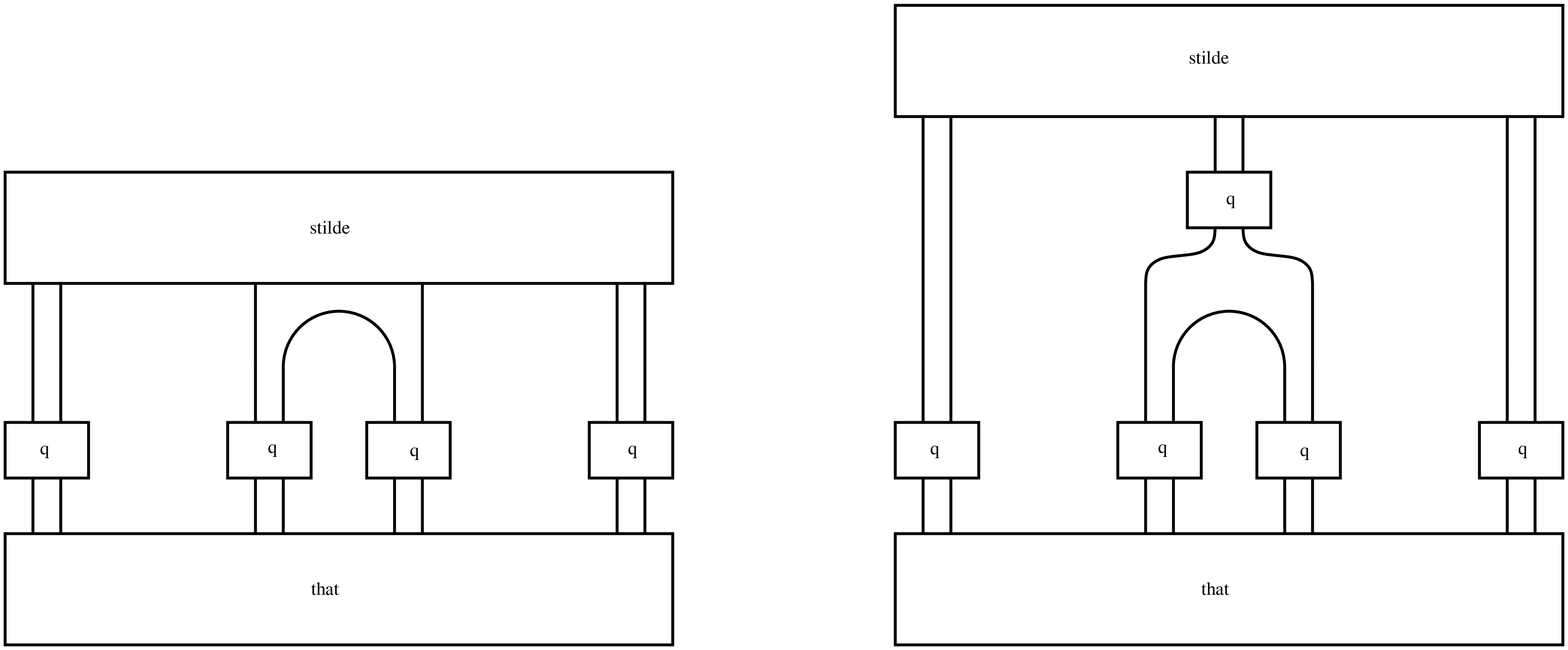}}
\end{center}
\caption{}\label{last}
\end{figure}

The two middle $q$'s in the picture on the right may be deleted using one application of Lemma \ref{follows} to $\hat{T}$ and then what is left is clearly
$S \circ_1 F(\tilde{S})$.

\begin{lemma}\label{follows}
Let $T$ be a $k$-tangle and $[k] = \{1,2,\cdots,k\}$ be regarded as the set of black external boundary arcs of $T$, enumerated, say, in clockwise direction starting from the one immediately next (counterclockwise)
to the $*$ arc.
Let $A \subseteq [k]$ be such that any black region of $T$ intersects at most one element of
$A$. Surround $T$ with $q$'s in all positions except those
given by $A$ and call this partially labelled tangle $F_A(T)$. Then $Z_{F_A(T)} = Z_{F(T)}$ on $P^\prime$.
\end{lemma}

\begin{proof} Consider the external boundary of any black region of $T$ that intersects an external boundary arc. Say it looks like something in Figure \ref{lemma2fig2}.
\begin{figure}[!h]
\begin{center}
\psfrag{&}{\Huge $\&$}
\resizebox{5cm}{!}{\includegraphics{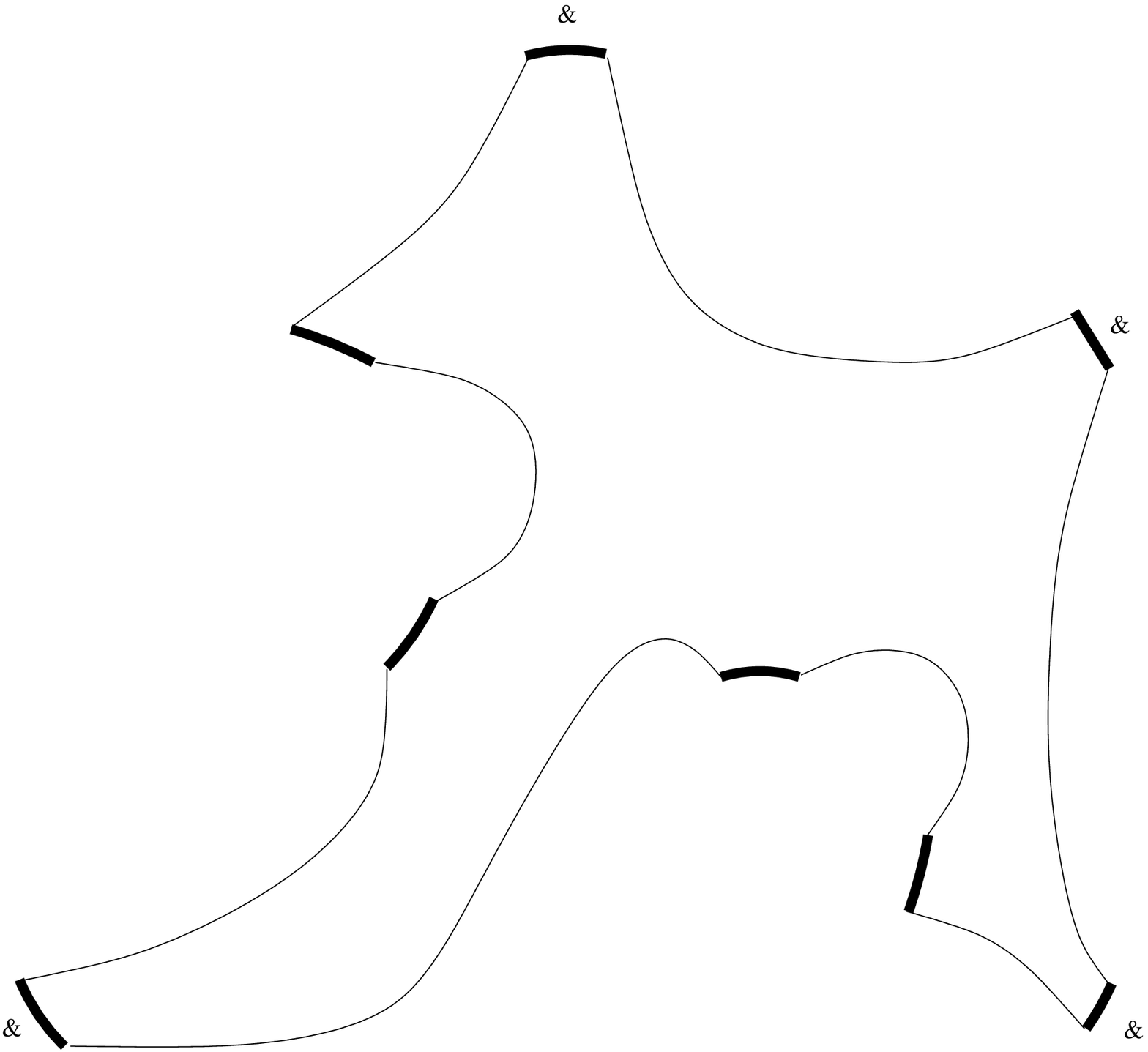}}
\end{center}
\caption{}\label{lemma2fig2}
\end{figure}
Here the dark portions represent boundary arcs of discs of $T$ while the light portions represent strings.
Say the portions marked $\&$ are boundary arcs of the external disc of $T$ while the rest are boundary arcs of various internal discs of $T$. By assumption, at most one of the portions marked $\&$ is in $A$.

Now in calculating $F(T)$, this black region looks as in Figure \ref{lemma2fig3}, where every 2-box has
a $q$ in it.
\begin{figure}[!h]
\begin{center}
\psfrag{&}{\Huge $\&$}
\resizebox{5cm}{!}{\includegraphics{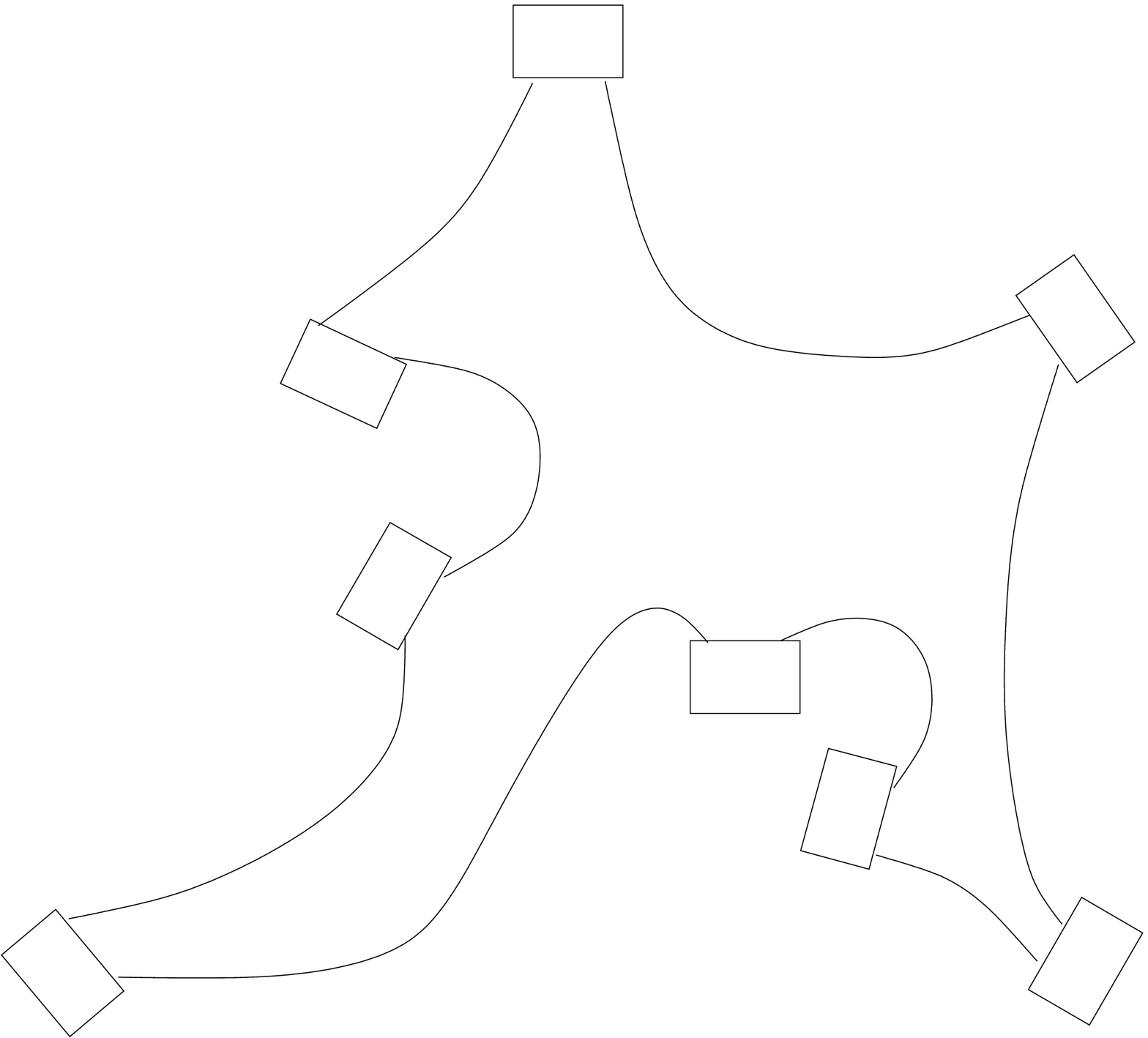}}
\end{center}
\caption{}\label{lemma2fig3}
\end{figure}
The external portions have a $q$ by definition of $F(T)$, while the internal portions have a $q$ because we're only interested in the values of the tangle when inputs come from $P^\prime$.
Now observe that in calculating $F_A(T)$, at most one of the $q$'s is missing - which does not matter
because of the exchange relation that $q$ satisfies (see Corollary \ref{cor:exchange}).
\end{proof}

\noindent
 Subcase 4.2(b): The black intervals $[2i-1,2i]$ and $[2i+1,2i+2]$ are part of the same black region in $\hat{T}$.
 
 Draw a dotted line from the midpoint of the interval $[2i-1,2i]$ to the midpoint of the interval 
$[2i+1,2i+2]$ in $\hat{T}$ that lies entirely in the black region that these are both part of. This
line does not intersect any string of $\hat{T}$ (by definition of a region) and so the part of $\hat{T}$ that lies inside this dotted line is a 1-box that joins the points $2i$ and $2i+1$. By irreducibility we may replace this one box by a scalar times a string and thus assume that
in $\hat{T}$ too, the points $2i$ and $2i+1$ are joined together. Thus $\hat{T}$ is of the form
in Figure \ref{subcase4.2b}
\begin{figure}[!h]
\begin{center}
\psfrag{cdots}{\huge ${\cdots}$}
\psfrag{stilde}{\huge $W$}
\psfrag{q}{\Huge $q$}
\psfrag{1}{\large $1$}
\psfrag{2i-1}{\large $2i$}
\psfrag{2i}{\large $2i+1$}
\psfrag{2i+1}{\large $2i+2$}
\psfrag{2i+2}{\large $2i+3$}
\psfrag{2i-2}{\large $2i-1$}
\psfrag{2k_1}{\large $2k_1$}
\psfrag{=}{\Huge $= \delta \tau(q)$}
\resizebox{5cm}{!}{\includegraphics{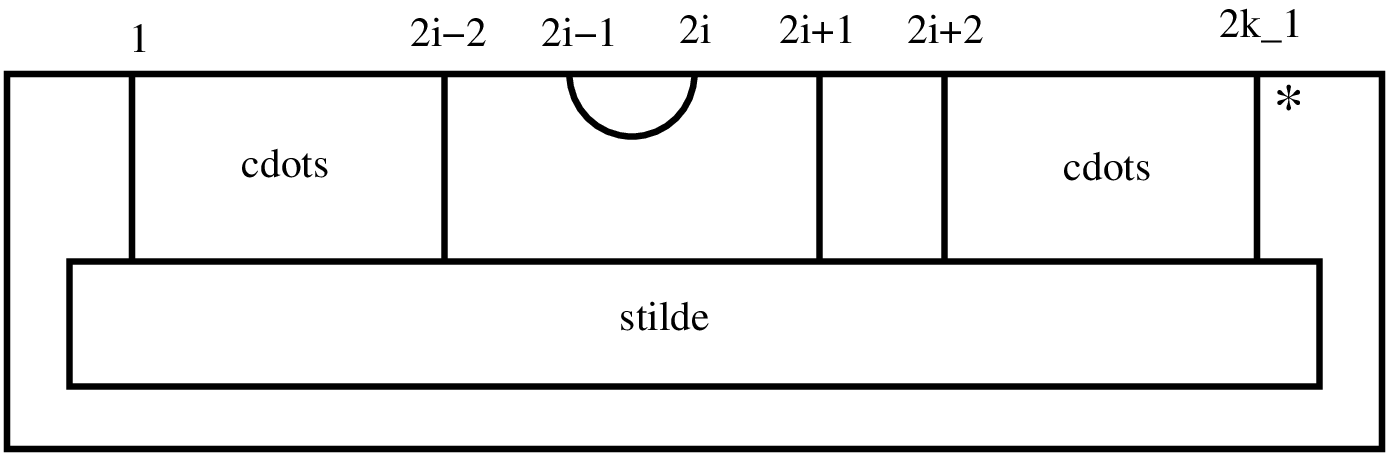}}
\end{center}
\caption{}\label{subcase4.2b}
\end{figure}

\noindent
for some tangle $W$ of colour $k_1-1$. Set $S$ to be the tangle in Figure \ref{stangle}.
\begin{figure}[!h]
\begin{center}
\psfrag{hatt}{\huge $W$}
\psfrag{k}{\huge $k_i-1$}
\psfrag{1}{\huge $1$}
\resizebox{3cm}{!}{\includegraphics{reduction2.eps}}
\end{center}
\caption{}\label{stangle}
\end{figure}

Again we claim that the truth of the statement of $S$ and $\tilde{S}$ implies that of the statement for $T$ and $\tilde{T}$. So suppose that
\begin{equation*}
Z_{S \circ_1 F(\tilde{S})} = \tau(q)^{\frac{1}{2}({k}_1-1 + l(S \circ_1 \tilde{S}) - l(S) - l(\tilde{S}))} Z_{S \circ_1 \tilde{S}}. 
\end{equation*}
Note that $T\circ_1 \tilde{T} = S \circ_1 \tilde{S}$ with one extra floating loop and therefore
$Z_{T \circ_1 \tilde{T}} = \delta Z_{S \circ_1 \tilde{S}}$ and $l({T \circ_1 \tilde{T}}) = l({S \circ_1 \tilde{S}}) +1$. Also $l(T) = l(\hat{T}) = l(W) = l(S)$ and we recall that $l(\tilde{S}) = l(\tilde{T})$.

It remains to compare $S \circ_1 F(\tilde{S})$ and $T \circ_1 F(\tilde{T})$. Observe that $T \circ_1 F(\tilde{T})$ equals the picture on the left in Figure \ref{finish} which equals $\delta \tau(q)$ times picture on the right
using properties of $q$ - which is clearly $S \circ_1 F(\tilde{S})$.

\begin{figure}[!h]
\begin{center}
\psfrag{that}{\huge $W$}
\psfrag{cdots}{\huge $\cdots$}
\psfrag{k}{\huge $k_i$}
\psfrag{q}{\huge $q$}
\psfrag{stilde}{\huge $\tilde{S}$}
\resizebox{12cm}{!}{\includegraphics{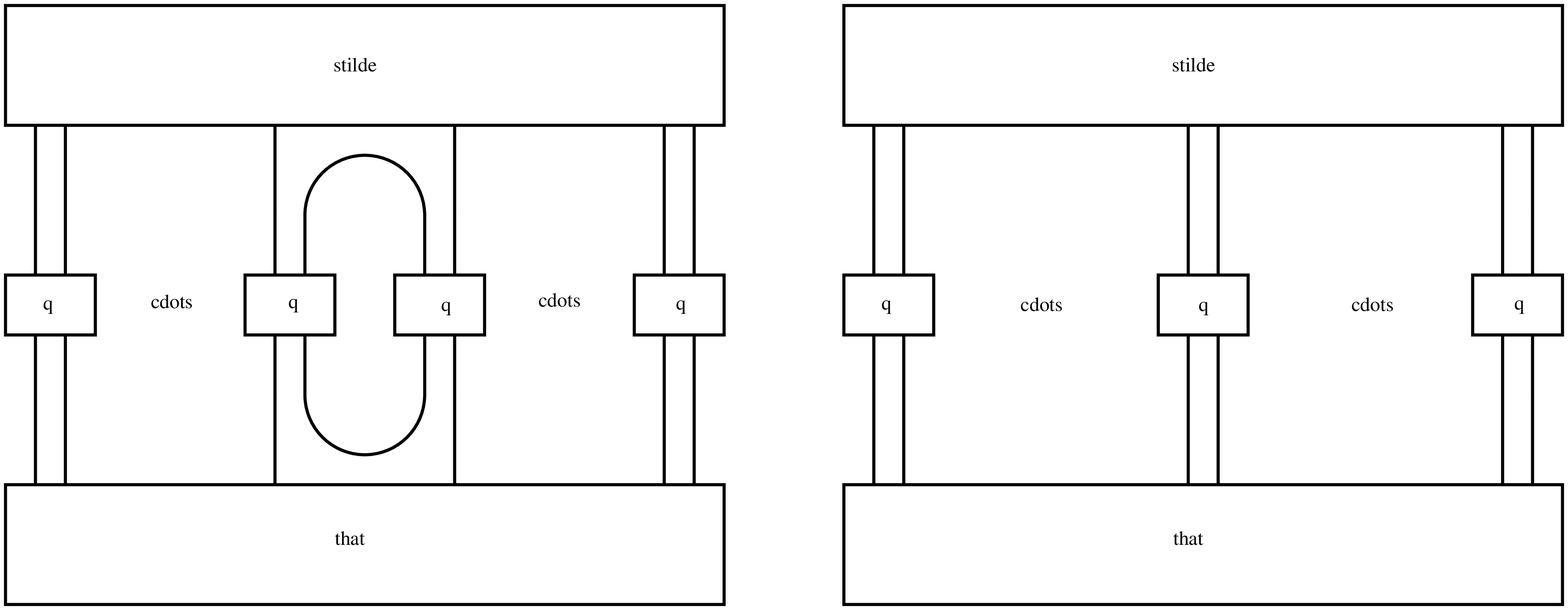}}
\end{center}
\caption{}\label{finish}
\end{figure}

Therefore $Z_{T \circ_1 F(\tilde{T})} = \delta \tau(q) Z_{S \circ_1 F(\tilde{S})}$. It now follows that 
\begin{equation*}
Z_{T \circ_1 F(\tilde{T})} = \tau(q)^{\frac{1}{2}({k}_1 + l(T \circ_1 \tilde{T}) - l(T) - l(\tilde{T}))} Z_{T \circ_1 \tilde{T}}. 
\end{equation*} This completes the proof of Theorem \ref{main}.

\bigskip
We  proceed to verify that our prescription for the tangle action does indeed specify various compatibility requirements that must be satisfied in order to define a planar algebra.

\bigskip

\noindent \textbf{(1) Compatibility with renumbering}

Let, $\sigma\in {\varSigma}_b$. Consider the tangle $\sigma(T)$ which as a subset
of ${\mathbb{R}}^2$ is the same as $T$ except that its $\sigma(i)$-th 
disc is the $i$-th disc of $T$.
We have to show the following diagram commutes:
\begin{displaymath}
 \xymatrix{
 P^{\prime}_{k_1} \otimes P^{\prime}_{k_2} \otimes \cdots \otimes P^{\prime}_{k_b} \ar[r]^-{U_\sigma} \ar[d]^{Z^{\prime}_T} & 
 P^{\prime}_{k_{\sigma^{-1}(1)}} \otimes P^{\prime}_{k_{\sigma^{-1}(2)}} \otimes \cdots 
 \otimes P^{\prime}_{k_{\sigma^{-1}(b)}} \ar[ld]^{Z^{\prime}_{\sigma(T)}}  & \\
 P^{\prime}_{k_0} }  
\end{displaymath}
Where,
$$
U_{\sigma}(x_1\otimes \cdots \otimes x_b)= 
x_{{\sigma}^{-1}(1)} \otimes \cdots \otimes x_{{\sigma}^{-1}(b)}
$$
for $x_i$ belongs to $P^\prime_i \subseteq P_i$. Now,
\begin{align*}
 & Z^\prime_{\sigma(T)}\circ U_{\sigma}(x_1\otimes \cdots \otimes x_b)\\
 & \qquad = Z^\prime_{\sigma(T)}(x_{{\sigma}^{-1}(1)} \otimes \cdots \otimes x_{{\sigma}^{-1}(b)})\\
 & \qquad = \alpha(\sigma(T)) F_{k_0}(Z_{\sigma(T)}(x_{{\sigma}^{-1}(1)} \otimes \cdots \otimes x_{{\sigma}^{-1}(b)}))\\
 & \qquad = \alpha(T) F_{k_0}(Z_{\sigma(T)} \circ U_{\sigma}(x_1 \otimes \cdots \otimes x_b))~~~~ \textrm{[since}~~ ~~~\alpha(\sigma(T))=\alpha(T)]\\
 & \qquad = \alpha(T) F_{k_0}(Z_T(x_1 \otimes \cdots \otimes x_b))~~\textrm{[~~renumbering ~~~axiom~~~for~~~Z]}\\
 & \qquad = Z^{\prime}_T(x_1 \otimes \cdots \otimes x_b)~~ \textrm{~~[by ~~definition]}
\end{align*}

\noindent \textbf{(2) Non-degeneracy}

We have to show, $ Z^{\prime}_{I^k_k} = id_{{P_k}^{\prime}}.$ Now 
for $x \in P^{\prime}_k$,
\begin{align*}
 & Z^{\prime}_{I^k_k} (x)\\
 & \qquad = \alpha(I^k_k)F_k(Z_{I^k_k}(x))~~~~\textrm{[by ~~~~definition]}\\
 & \qquad = \alpha(I^k_k) F_k(x)~~~~\textrm{[non-degeneracy~~~~of~~Z]}\\
 & \qquad = F_k(x)~~~~~~~~~~\textrm{[since~~~~}~~~\alpha(I^k_k)= 1]\\
 & \qquad = x
\end{align*}
\noindent \textbf{(3) Compatibility with respect to substitution}

Let $T= T^{k_0}_{k_1,\cdots,k_b}$ and $ \widetilde{T}= T^{{\widetilde{k}}_0}_{\widetilde{k_1},\cdots, {\widetilde{k}}_{\tilde{b}}}$ with $\tilde{k_0}= k_i$ for some $i\in \{1,\cdots,b\}$.
We need to check that the following diagram commutes:\newline
When $\tilde{b} > 0 :$
 \begin{displaymath}
  \xymatrix{
  (\otimes_{j=1}^{i-1}P^{\prime}_{k_j}) \otimes (\otimes_{j=1}^{\widetilde{b}} P^{\prime}_{\widetilde{k}_j}) \otimes
  (\otimes_{j=i+1}^{b}P^{\prime}_{k_j}) \ar[d]_-{(\otimes_{j=1}^{i-1}id_{P^{\prime}_{k_j}}) \otimes Z^{\prime}_{\widetilde{T}} \otimes
  (\otimes_{j=i+1}^{b}id_{P^{\prime}_{k_j}})} \ar[dr]^-{Z^{\prime}_{T \circ_i {\widetilde{T}}}} &\\
  (\otimes_{j=1}^{b}P^{\prime}_{k_j}) \ar[r]^{Z^{\prime}_T} &
  P^{\prime}_{k_0} }  
 \end{displaymath}
Let $({\otimes}^{i-1}_{j=1} x_j) \otimes ({\otimes}^{\widetilde{b}}_{j=1} \widetilde{x_j}) \otimes ({\otimes}^b_{j=i+1} x_j)$ belongs to
$({\otimes}^{i-1}_{j=1} P^{\prime}_{k_j}) \otimes ({\otimes}^{\widetilde{b}}_{j=1} P^{\prime}_{\widetilde{k_j}}) \otimes ({\otimes}^b_{j=i+1}  P^{\prime}_{k_j})$
Then,
\begin{align*}
& Z^{\prime}_T \circ (id \otimes Z^{\prime}_{\widetilde{T}} \otimes id ) [({\otimes}^{i-1}_{j=1} x_j) \otimes ({\otimes}^{\widetilde{b}}_{j=1} \widetilde{x_j}) \otimes ({\otimes}^b_{j=i+1} x_j)]\\
& \qquad = Z^{\prime}_T[({\otimes}^{i-1}_{j=1} x_j) \otimes Z^{\prime}_{\widetilde{T}}({\otimes}^{\widetilde{b}}_{j=1} \widetilde {x_j}) \otimes ({\otimes}^b_{j=i+1}
 x_j)]\\
& \qquad = Z^{\prime}_T[({\otimes}^{i-1}_{j=1} x_j) \otimes \alpha(\widetilde{T}) Z_{E \circ \widetilde{T}}(({\otimes}^{\widetilde{k_0}} q)
\otimes ({\otimes}^{\widetilde{b}}_{j=1} \widetilde{x_j})) \otimes ({\otimes}^b_{j=i+1} x_j)]\\
&\qquad \qquad \qquad \qquad \qquad \qquad \qquad \qquad \qquad \qquad \textrm{~~~~~(by~~definition}~~of Z^{\prime}_T)\\
& \qquad = \alpha(T)\alpha(\widetilde{T}) Z_{E \circ T}[({\otimes}^{k_0} q) \otimes ({\otimes}^{i-1}_{j=1} x_j) \otimes Z_{E \circ \widetilde{T}}
 (({\otimes}^{\widetilde{k_0}} q) \otimes ({\otimes}^{\widetilde{b}}_{j=1} {\widetilde{x}}_j)) \otimes ({\otimes}^b_{j=i+1} x_j)]\\
& \qquad \qquad \qquad \qquad \qquad \qquad \qquad \qquad \qquad \qquad \textrm{~~~~~(by~~definition}~~of Z^{\prime}_T)\\
& \qquad = \alpha(T)\alpha(\widetilde{T}) Z_{E\circ(T\circ_i(E\circ \widetilde{T}))}[({\otimes}^{k_0} q) \otimes ({\otimes}^{i-1}_{j=1} x_j)
\otimes ({\otimes}^{\widetilde{k_0}} q) \otimes ({\otimes}^{\widetilde{b}}_{j=1} {\widetilde{x}}_j) \otimes ({\otimes}^b_{j=i+1} x_j)]\\
& \qquad \qquad \qquad \qquad \qquad \qquad \qquad \qquad \qquad \qquad \text{(since $Z$ is associative)}\\
& \qquad = \alpha(T) \alpha(\widetilde{T}) \frac{\alpha(T\circ_i \widetilde{T})}{\alpha(T) \alpha(\widetilde{T})}
Z_{E\circ(T\circ_i{\widetilde{T}})}[({\otimes}^{k_0} q) \otimes ({\otimes}^{i-1}_{j=1} x_j) \otimes ({\otimes}^{\widetilde{b}}_{j=1}
 {\widetilde{x}_j}) \otimes ({\otimes}^b_{j=i+1} x_j)]\\
& \qquad \qquad \qquad \qquad \qquad \qquad \qquad \qquad \qquad \qquad \text{(by Theorem \ref{main})}\\
& \qquad = Z^{\prime}_{T\circ_i{\widetilde{T}}}[({\otimes}^{i-1}_{j=1} x_j) \otimes ({\otimes}^{\widetilde{b}}_{j=1} {\widetilde{x}_j})
 \otimes ({\otimes}^b_{j=i+1} x_j)] \textrm{~~~~~~~~~~(by~~definition)}.\\
\end{align*} 
When $\tilde{b} = 0 :$ We need to check the following diagram commutes:\newline
\begin{displaymath}
  \xymatrix{
  (\otimes_{j=1}^{i-1}P^{\prime}_{k_j}) \otimes \mathbb{C} \otimes
  (\otimes_{j=i+1}^{b}P^{\prime}_{k_j})\ar[r]^-\cong \ar[d]_-{(\otimes_{j=1}^{i-1}id_{P^{\prime}_{k_j}})} \otimes Z^{\prime}_{\widetilde{T}} \otimes
  (\otimes_{j=i+1}^{b}id_{P^{\prime}_{k_j}}) & 
  \otimes_{\substack{j=1, \\ j\neq i}}^{b}P^{\prime}_{k_j}
  \ar[d]_-{Z^{\prime}_{T \circ_i {\widetilde{T}}}} & \\
  \otimes_{j=1}^{b}P^{\prime}_{k_j} \ar[r]^{Z^{\prime}_T} &
  P^{\prime}_{k_0} }  
\end{displaymath}
 The proof is as above.\newline
Thus $T \mapsto Z^{\prime}_T$ is compatible with substitution.

In conclusion, the collection $P^{\prime}=\{P^{\prime}_k: k\in Col\}$
of vector spaces, equipped with the assignment $T\mapsto Z^{\prime}_T$ of multilinear maps is a planar algebra.\\ 

 \begin{proof}{\em (of part (b) in the notation of the paragraph following the statement of Theorem \ref{main}.)}\\

  That $P^{\prime}_k=(P^{(N \subseteq Q)})_k$ and consisteny under inclusions of two sides follows from definition. That  $(P^{\prime},Z^{\prime})$
   has modulus $\sqrt{[Q:N]}$ follows from definition of $\alpha$. 
  We need, further, to show the following:
  \begin{fact}
  $(i) Z^{\prime}_{{\mathcal{E}}^{2k+1}}(1)= \sqrt{[Q:N]} e^Q_{2k}$ and
  $(ii) Z^{\prime}_{{\mathcal{E}}^{2k}} (1)= \sqrt{[Q:N]} e^Q_{2k-1}.$
  \end{fact}
  See the figure \ref{fig:jp1}. Left one is for case (ii) and right one is for case (i).\\
  \begin{figure}[h]
 \includegraphics[scale=.7]{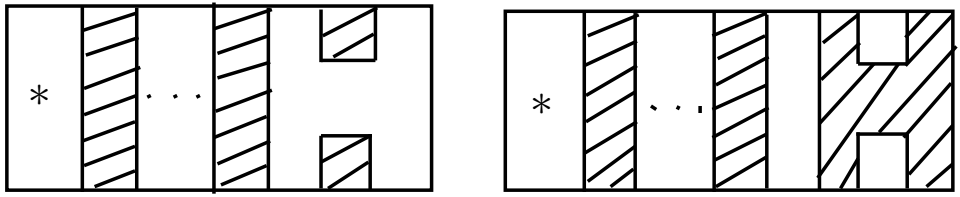}
 \caption{Jones Projection}
 \label{fig:jp1}
\end{figure}
\underline{Justification of (i)}: $\alpha({\mathcal{E}}^{2k+1})= \sqrt{[M:Q]}$. \\By definition,
\begin{align*}
 & Z^{\prime}_{{\mathcal{E}}^{2k+1}}(1)\\
 & \qquad = \sqrt{[M:Q]} F_{2n+1}(Z^{N\subseteq M}_{{\mathcal{E}}^{2k+1}}(1))\\
 & \qquad = \sqrt{[M:Q]} \hspace{5mm} {\begin{minipage}{.4\textwidth}
    \centering
    \includegraphics[scale=.6]{jopo1.eps}
    \end{minipage}}\\
 & \qquad = \sqrt{[Q:N]} e^Q_{2k} \textrm{~~~[by ~~Theorem~~\ref{planarstinv}]}
\end{align*}
This justifies the fact(i).\\
\underline{Justification of (ii)}: $\alpha({\mathcal{E}}^{2k})= [M:Q]^{-\frac{1}{2}}$. \\By definition,
\begin{align*}
 & Z^{\prime}_{{\mathcal{E}}^{2k}}(1)\\
 & \qquad = [M:Q]^{-\frac{1}{2}} F_{2k}(Z^{N\subseteq M}_{{\mathcal{E}}^{2k}}(1))\\
 & \qquad =[M:Q]^{-\frac{1}{2}}  {\begin{minipage}{.4\textwidth}
    \centering
    \includegraphics[scale=.6]{jopo2.eps}
    \end{minipage}}\\
 & \qquad = \sqrt{[Q:N]} e^Q_{2k-1} \textrm{~~~[by ~~~Theorem~~\ref{planarstinv}]}
\end{align*}
This justifies the fact(ii).
\begin{fact}
 $ Z^{\prime}_{{(E^{\prime})}^n_n} (x) = \sqrt{[Q:N]} E_{Q^{\prime}\cap {Q_{n-1}}} (x)$ for all $x$ belongs to $N^{\prime}\cap Q_{n-1}$
  and $k\geq 1$. Where corresponding trace of $E^{N^{\prime}\cap Q_{n-1}}_{Q^{\prime}\cap Q_{n-1}}$ is $tr_{N\subseteq Q}.$ 
  \end{fact}
  \underline{Justification}:
  The tangles ${(E^{\prime})}^n_n$ are as in figure \ref{LCE} according as $n$ is odd or even respectively.
  \begin{figure}[h]
   \includegraphics[scale=.5]{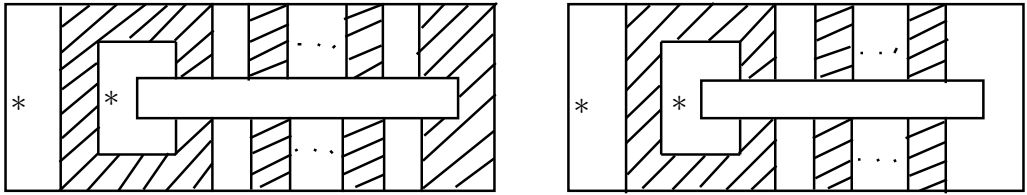}
   \caption{Left Conditional Expectation}
   \label{LCE}
  \end{figure}\\
  Consider the case when $n$ is odd. 
  Now, for all $y \in M^{\prime}\cap M_{n-1}$,\\ $tr([M:Q]F(E^{N^{\prime}\cap M_{n-1}}_{M^{\prime}\cap M_{n-1}} (x) F(y)))$ is equal to
  \begin{align*}
  & {\delta}^{-n}\frac{[M:Q]}{\sqrt{[M:N]}} \hspace{10mm} {\begin{minipage}{.5\textwidth}
    \centering
    \includegraphics[scale=.35]{jt1.eps}
    \end{minipage}}\\\\
   & \qquad = {\delta}^{-n}\frac{[M:Q]}{\sqrt{[M:N]}}  {\begin{minipage}{.5\textwidth}
    \centering
    \includegraphics[scale=.35]{jt2.eps}
    \end{minipage}}\\\\
  &\qquad\qquad\qquad\qquad\qquad \qquad~~~~\qquad\qquad~~~~\textrm~~{~~[by ~~~Theorem~~~ \ref{Bisch}~~~(a)]}\\\\
    & \qquad = {\delta}^{-n}\frac{[M:Q]}{\sqrt{[M:N]}}  {\begin{minipage}{.5\textwidth}
    \centering
    \includegraphics[scale=.35]{jt3.eps}
    \end{minipage}}\\\\
  & \qquad\qquad\qquad\qquad\qquad \qquad\textrm~~{~~~[since~~~x~~\in F_n(P_n)}~~\textrm{and~~by~~exchange ~~relation]}\\\\
   & \qquad  = {\delta}^{-n}\frac{[M:Q]}{\sqrt{[M:N]}}  {\begin{minipage}{.5\textwidth}
    \centering
    \includegraphics[scale=.35]{jt4.eps}
    \end{minipage}}\\\\
  &\qquad\qquad\qquad\qquad\qquad\qquad\qquad\qquad~~~\textrm{~~[by~~extremality]}\\\\
   &\qquad \qquad = {\delta}^{-n} {\begin{minipage}{.5\textwidth}
    \centering
    \includegraphics[scale=.35]{jt5.eps}
    \end{minipage}}\\\\
 &\qquad\qquad\qquad\qquad\qquad\qquad ~~\textrm{~~[by~~Theorem~~\ref{Bisch}~~(c)~~~and~~~ ~~~extremality]}\\\\\\\\
   &\qquad \qquad = {\delta}^{-n}{\begin{minipage}{.5\textwidth}
    \centering
    \includegraphics[scale=.35]{jt6.eps}
    \end{minipage}}\\\\
    & \qquad \qquad\qquad\qquad\qquad\qquad~~~\textrm{~~[as~~x}~~\in~~F_n(P_n)]\\\\
   & \qquad \qquad= tr(xF(y)).
  \end {align*}
  Thus it follows that 
  \begin{equation}\label{eqce}
   E^{F(N^{\prime}\cap M_{n-1})}_{F(M^{\prime}\cap M_{n-1})}(x)= [M:Q] F(E^{N^{\prime} \cap M_{n-1}}_{M^{\prime}\cap M_{n-1}}(x)).
  \end{equation}
Now,
\begin{align*}
&  Z^{\prime}_{{(E^{\prime})}^n_n} (x) \\
 & \qquad = \alpha ({(E^{\prime})}^n_n) F(Z_{{(E^{\prime})}^n_n}(x))~~\textrm~~{[by~~definition~~]}\\
 & \qquad = [M:Q]^{\frac{1}{2}} [M:N]^{\frac{1}{2}} F(E^{N^{\prime}\cap M_{n-1}}_{M^{\prime}\cap M_{n-1}}(x))~~\textrm~~{[by~~~~~definition~~of~~~\alpha]}\\\
 & \qquad = [Q:N]^{\frac{1}{2}} E^{F({N^{\prime}\cap M_{n-1}})}_{F(M^{\prime}\cap M_{n-1})}(x)~~~\textrm{~~[by ~~~Equation~~~~(\ref{eqce})]}
\end{align*}
This completes the proof for odd case. Even case is exactly similar, so we omit it.
  \begin{fact}
  $  Z^{\prime}_{E^n_{n+1}} (x)= \sqrt{[Q:N]} E_{N^{\prime} \cap Q_{n-1}}(x)$ for all $x$ belongs to $N^{\prime}\cap Q_n$
  and this is required to hold for all $k$ in Col, where for $k=0_{\underline{+}}$, the equation is interpreted as
  $Z^{\prime}_{E^{0_{+}}_1}(x)= \sqrt{[Q:N]}tr_{N\subseteq Q}(x)$ for all $x$ belongs to $N^{\prime}\cap Q$.
  Here again the trace corresponding to the conditional expectation is given by $tr_{N \subseteq Q}$.
  \end{fact}
  \begin{figure}[h]
 \includegraphics[scale=.5]{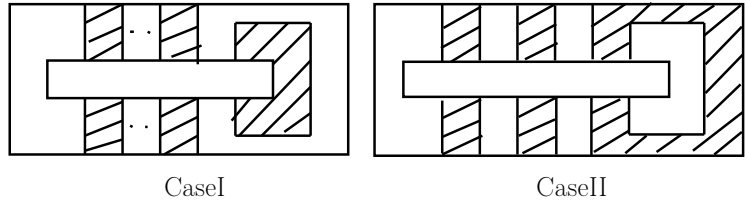}
 \caption{Conditional Expectation}
 \label{fig:co}
 \end{figure}
  \underline{Justification}:
  Consider the conditional expectation tangle as in Figure \ref{fig:co}. For Case I we give a diagramatic proof and for Case II we give analytic proof.
  \bigskip

  $\mathcal Case I (E^{2n}_{2n+1}) :$
  Now, for all $y \in F(N^{\prime}\cap M_{2n-1})$, $tr(F(E^{N^{\prime}\cap M_{2n}}_{N^{\prime}\cap M_{2n-1}} (x)) F(y))$ is equal to 
  \begin{align*}
   & \qquad {\delta}^{-2n} {\begin{minipage}{.4\textwidth}
    \centering
    \includegraphics[scale=.45]{jtr1.eps}
    \end{minipage}}\\\\\\\\
   & \qquad = {\delta}^{-2n}{\begin{minipage}{.4\textwidth}
    \centering
    \includegraphics[scale=.45]{jtr2.eps}
    \end{minipage}}\\\\
    & \qquad\qquad\qquad\qquad\qquad\qquad\qquad~~~\textrm{~~[by ~~Theorem~~\ref{Bisch}(a)]}\\\\\\
   &\qquad = {\delta}^{-2n} {\begin{minipage}{.4\textwidth}
    \centering
    \includegraphics[scale=.45]{jtr3.eps}
    \end{minipage}}\\\\
    & \qquad\qquad\qquad\qquad\qquad\qquad\qquad~~\textrm{~~~~~[since},~~y~~\in~~P_{2n}~~\subseteq~~P_{2n-1}]\\\\
   & \qquad = tr(xF(y))
  \end{align*}

Thus,
\begin{equation}\label{eqce2}
 E^{F(N^{\prime}\cap M_{2n})}_{F(N^{\prime}\cap M_{2n-1})}(x)= F(E^{N^{\prime}\cap M_{2n}}_{N^{\prime}\cap M_{2n-1}}(x)).
\end{equation}
\vspace{1cm}
Then the following equations hold:
\begin{align*}
 & Z^{\prime}_{E^{2n}_{2n+1}}(x)\\
 & \qquad = \alpha(E^{2n}_{2n+1}) F(Z_{E^{2n}_{2n+1}}(x))~~\textrm{~~[by~~definition~~]}\\
 & \qquad = [M:Q]^{-\frac{1}{2}} F(Z_{E^{2n}_{2n+1}}(x))~~\textrm{~~[by~~definition~~of}~~\alpha]\\
 & \qquad= [M:Q]^{-\frac{1}{2}} [M:N]^{\frac{1}{2}} F(E^{N^{\prime}\cap M_{2n}}_{N^{\prime}\cap M_{2n-1}}(x)) \\
 & \qquad = [Q:N]^{\frac{1}{2}} E^{F(N^{\prime}\cap M_{2n})}_{F(N^{\prime}\cap M_{2n-1})}(x)~~\textrm{~~[by~~Equation~~(\ref{eqce2})]}
 \end{align*}
 Thus, we proved $ Z^{\prime}_{{(E^{\prime})}^n_n} (x) = \sqrt{[Q:N]} E_{Q^{\prime}\cap {Q_{n-1}}} (x)$.\vspace{2mm} \\
$\mathcal Case II (E^{2n-1}_{2n}):$ By definition, $\alpha(E^{2n-1}_{2n})=\sqrt{[M:Q]}$. Then, for $x \in P^{\prime}_{2n}= F_{2n}(P_{2n})$,
\begin{equation}\label{star}
 Z^{\prime}_{E^{2n-1}_{2n}}(x)= \sqrt{[M:Q]} F(Z^{N\subseteq M}_{E^{2n-1}_{2n}}(x))= \sqrt{[M:Q]}\sqrt{[M:N]} F(E^{N^{\prime}\cap M_{2n-1}}_{N^{\prime}\cap M_{2n-2}}(x))
\end{equation}
Then we claim,
\begin{equation*}
 E^{F(N^{\prime}\cap M_{2n-1})}_{F(N^{\prime}\cap M_{2n-2})}(x) = [M:Q] F(E^{N^{\prime}\cap M_{2n-1}}_{N^{\prime}\cap M_{2n-2}}(x))
 \end{equation*}
\underline{Justification}: 
Firstly observe (see \cite{BhaLa}),
\begin{equation*}
 F_{2n}(P_{2n})= p_{[0,{2n-1}]} (N^{\prime} \cap M_{2n-1}) p_{[0,{2n-1}]}.
 \end{equation*}
Put, $ x= p_{[0,{2n-1}]} m_{2n-1} p_{[0,{2n-1}]}$ for $ m_{2n-1} \in N^{\prime}\cap M_{2n-1}$.
Then the following  self-explanatory array of equations
hold for any $m_{2n-2} \in N^{\prime}\cap M_{2n-2} $(using Fact \ref{f:erel} repeatedly): 
\begin{align*}
 & tr_{N\subseteq Q}([M:Q] p_{[0,{2n-1}]} E^{N^{\prime}\cap M_{2n-1}}_{N^{\prime}\cap M_{2n-2}}(p_{[0,{2n-1}]} m_{2n-1} p_{[0,{2n-1}]})  p_{[0,{2n-1}]}  p_{[0,{2n-1}]} m_{2n-2} p_{[0,{2n-1}]})\\
 & \qquad = {[M:Q]}^{n+1} tr(p_{[0,{2n-1}]} E^{N^{\prime}\cap M_{2n-1}}_{N^{\prime}\cap M_{2n-2}}(p_{[0,{2n-1}]} m_{2n-1} p_{[0,{2n-1}]})  p_{[0,{2n-1}]} m_{2n-2} p_{[0,{2n-1}]})\\
  & \qquad = {[M:Q]}^{n+1} tr( E^{N^{\prime}\cap M_{2n-1}}_{N^{\prime}\cap M_{2n-2}}(p_{[0,{2n-1}]} m_{2n-1} p_{[0,{2n-1}]})p_{[0,{2n-1}]} m_{2n-2} p_{[0,{2n-1}]})\\
 & \qquad = {[M:Q]}^{n+1} tr( E^{N^{\prime}\cap M_{2n-1}}_{N^{\prime}\cap M_{2n-2}}(p_{[0,{2n-1}]} m_{2n-1} p_{[0,{2n-1}]})p_{[0,{2n-3}]}E^{N^{\prime}\cap M_{2n-2}}_{N^{\prime}\cap P_{2n-2}}( m_{2n-2}) p_{[0,{2n-3}]}e_{0,{2n-1}})\\
 & \qquad = {[M:Q]}^{n+1} tr( E^{N^{\prime}\cap M_{2n-1}}_{N^{\prime}\cap M_{2n-2}}(p_{[0,{2n-1}]} m_{2n-1} p_{[0,{2n-1}]} E^{N^{\prime}\cap M_{2n-2}}_{N^{\prime}\cap P_{2n-2}}( m_{2n-2}) p_{[0,{2n-3}]})e_{0,{2n-1}})\\
 & \qquad = {[M:Q]}^n tr(p_{[0,{2n-1}]} m_{2n-1} p_{[0,{2n-1}]} E^{N^{\prime}\cap M_{2n-2}}_{N^{\prime}\cap P_{2n-2}}( m_{2n-2}) p_{[0,{2n-3}]})~~\textrm{~~~[Markov Property]}\\
  & \qquad = {[M:Q]}^n tr(p_{[0,{2n-1}]} m_{2n-1} p_{[0,{2n-1}]} p_{[0,{2n-1}]} m_{2n-2} p_{[0,{2n-1}]})\\
  & \qquad =  tr_{N\subseteq Q}(p_{[0,{2n-1}]} m_{2n-1} p_{[0,{2n-1}]} p_{[0,{2n-1}]} m_{2n-2} p_{[0,{2n-1}]})\\
\end{align*}
Thus from definition of trace preserving conditional expectation we conclude that the claim is justified. Hence from Equation (\ref{star}) it follows that,
\begin{equation*}
 Z^{\prime}_{E^{2n-1}_{2n}}(x)= \frac{\sqrt{[M:Q]} \sqrt{[M:N]}}{[M:Q]} E^{F(N^{\prime}\cap M_{2n-1})}_{F(N^{\prime}\cap M_{2n-2})}(x)= \sqrt{[Q:N]} E^{P^{\prime}_{2n}}_{P^{\prime}_{2n-1}}(x)
\end{equation*}
This is what we wanted to show.
\end{proof}
Thus the proof of Theorem \ref{main1} is complete.

\section{Examples}
\subsection{Dual Intermediate Planar algebra}

We now consider the other intermediate subfactor $Q  \ss M$.
We can describe its planar algebra using Theorem~\ref{main1} and the
fact that it is the dual subfactor to $M \ss Q _1$.  Namely, apply
Theorem~\ref{main1} to the planar algebra $(P_{1,n}(L))_n$ of $M \ss M_1$,
with respect to the projection
$[M:Q ]^{1/2}[Q :N]^{-1/2} \begin{minipage}{.1\textwidth}
  \centering
  \includegraphics[scale= .3]{q1.eps}
 \end{minipage}$.  We obtain the planar algebra
$(P^{(M\ss Q _1)}_n)_n$.  The planar algebra of $Q  \ss M$ is its dual,
$(P^{(M\ss Q _1)}_{1,n})_n$.  If we carry out this process, we obtain the
following planar algebra:\\

\begin{definition}
 Denote the following tangle by $E^{\prime}_n$:
  \begin{center} \includegraphics[scale=.5]{dual2.eps} \hspace{1in} for $n$ odd, \end{center}
\begin{center} \includegraphics[scale=.5]{dual1.eps} \hspace{1in} for $n$ even. \end{center}

where ${\begin{minipage}{.1\textwidth}
    \centering
    \includegraphics[scale=.55]{r.eps}
    \end{minipage}}= \sqrt{\frac{[M:Q]}{[Q:N]}} ~~{\begin{minipage}{.1\textwidth}
    \centering
    \includegraphics[scale=.55]{q.eps}
    \end{minipage}}$

according as $n$ is  odd  or even respectively. We shall use these to define a map $T \mapsto G(T)$ from the class of $k$-tangles to the class of {\em partially labelled} $k$-tangles with $(k+1)$ internal discs all but the last of which are 2-boxes labelled with a $r$, with the the tangle $T$ inserted in the last disc of colour $k$. Thus, $G(T) = E^{\prime}_k\circ_{ (D_1,D_2,\cdots, D_k, D_{k+1})}(r,r,\cdots,r,T)$.

If it is clear from the context then we write $E^{\prime}$ instead of $E^{\prime}_n$. \par Define functions $G_n :P_n \mapsto P_n$ by $G_n(x)=Z_{E^{\prime}_n}(r\otimes r \otimes \cdots \otimes r \otimes x)$
for $x \in P_n$. We often write $G(x)$ instead of $G_n(x)$ if there is no confusion.
\end{definition}
\begin {definition}
  Let $ T $ be a $k$-tangle with $ b\geq 1 $ internal discs ${D_1,\dots D_b}$ of colours ${k_1,\dots k_b}$. Then define $\widetilde{\alpha}(T) = [Q:N]^{\frac{1}{2}\widetilde{c}(T)}$, where
\[\widetilde{c}(T)=(\lceil k_0/2 \rceil+\lfloor k_1/2 \rfloor+\dots +\lfloor k_b/2 \rfloor)-\widetilde{l}(T) \]
with $\widetilde{l}(T)$ being the number of closed loops after capping the white intervals of the external disc  of $T$ and cupping the white intervals of all internal discs of $T$. 
\end {definition}

It is straightforward to verify the following corollary:
\begin{corollary}
\label{main2} If $P^{\prime \prime}_k= ran(G(I^k_k))$ and $Z^{\prime \prime}_T = \widetilde{\alpha}(T) Z_{G(T)|_{\otimes P^{\prime \prime}_{k_i(T)}}}$, then $(P^{\prime \prime}, T \mapsto Z^{\prime \prime}_T|_{\otimes P^{\prime \prime}_{k_i(T)}})$ is a subfactor planar algebra which is isomorphic to $P^{(Q\ss M)}$. 
  \end{corollary}
  \bigskip

\subsection{Crossed Product Example}\label{CPE}
Landau described  the planar algebra $P(G)$ of the group subfactor (corresponding to the fixed-points of an outer action
of a finite group $G$ on a $II_1$ factor ) which has a presenatation with generators given by $L_2= G$ and $L_k=\phi$ for $k\neq2,$
and the relation that a simple closed loop of either colour be the scalar $\sqrt{\lvert G \rvert}$ and the additional six relations
labelled $00,0,1,2,3,4$ as in \cite{La1}. Denote by $e$, the identity of $G$.

\smallskip
We have another group $\Theta$ and an action $\alpha:\Theta \mapsto Aut(G)$ as in \cite{LaSu}. Without loss of generality we can assume 
$\alpha$ is $1-1.$ Denote by  $f$, the identity of $\Theta$. The map that replaces the label of each $2$-box with the label's image under $\theta\in \Theta$
defines an automorphism of $P(G)$ (we will denote this also by $\theta$). Then the set $P^{\Theta}$ of invariants for this action is a
sub-planar algebra of $P$, and the set of $\Theta$-invariant $k$-boxes of $P(G)$ constituts precisely the set of $k$-boxes
of $P^{\Theta}$.

\smallskip
\begin{notation}\label{basis}
We follow the same notation as in \cite{LaSu} [Remark $3.3.1.(b)$] to denote an orthonormal basis of $P(G)_k$ (with respect to the inner product given by the natural trace):
define $S(\bar{g})$ (where $\bar{g} \in G^{k-1}$) to be the labelled $k$-tangle $(k >2)$ given by the following two Figures \ref{basis1} and \ref{basis2} for $k$ odd and even respectively.
 \begin{figure}
   \includegraphics[scale=.5]{cross1}
   \caption{$k$ Odd}
  \label{basis1}
  \end{figure}

\begin{figure}
   \includegraphics[scale=.5]{cross2}
   \caption{$k$ Even}
   \label{basis2}
  \end{figure}

Also,
\begin{equation*}
 S(g)= {\begin{minipage}{.2\textwidth}
    \centering
    \includegraphics[scale=.7]{gg}
    \end{minipage}}
\end{equation*}

\smallskip
We use Latin alphabets to denote the elements of $G$, whereas we  use Greek symbols to write the elements of $\Theta$. As usual we write the elements of $G\rtimes \Theta$ as ordered pairs $(g,\theta)$ with the ususal multiplication $(g_1,\theta_1)(g_2,\theta_2)= (g_1\theta_1(g_2),\theta_1\theta_2).$
Also for each integer $k\geq 1$ and $\theta \in \Theta$, we simply write $\theta(g_1,\cdots,g_k)$ to denote the map $\alpha^{(k)}_{\theta} \in Aut(G^k)$
defined by $\alpha^{(k)}_{\theta}(g_1,g_2,\cdots,g_k)= (\alpha_{\theta}(g_1),\alpha_{\theta}(g_2),\cdots,\alpha_{\theta}(g_k))$.  Lastly by ${\bar{\delta}}_n$ we denote the $n$-tuple
$(\delta_1,\cdots,\delta_n)$. If from the context it is obvious what is $n$,we simply write $\bar{\delta}.$ For convenience we denote by 
${\bar{\delta}}_{[k,n]} ($ respectively, ${\bar{\delta}}_{(k,n]}$) the tuple $(\delta_k,\cdots,\delta_n) ($ respectively, $(\delta_{k+1},\cdots,\delta_n)).$
\end{notation}
We prove the following theorem (\cite{LaSu}):
\begin{theorem}\label{subgroup}
 Let $G,\Theta$ be as above, and let $G\rtimes \Theta$ denote the semi-direct product,and let $N= R^{G\rtimes \Theta}\subset R^{\Theta}=M$
 denote the corresponding subgroup-subfactor. Then,
 $$P^{\Theta} \simeq P^{(N\subset M)}.$$
\end{theorem}

We prove this theorem in various steps.
\begin{fact}\label{s}
For $k\geq 3$, using exchange relation repeatedly and other relations labelled $0,1,2$ as stated in \cite{La1} we get,
\begin{align*}
 & S(g_1,g_2,\cdots,g_{k-1}) S(h_1,h_2,\cdots,h_{k-1})\\
 & \qquad = {(\sqrt{ \lvert G\rvert})}^{(\lceil k/2 \rceil-1)}(\prod_{i=2}^{\lceil k/2 \rceil} \delta(h_1 g_{k+1-i}, h_i))
  S(h_1g_1,h_1g_2,\cdots,h_1 g_{\lceil k/2 \rceil}, h_{\lceil k/2 \rceil +1},\cdots, h_{k-1})
\end{align*}
Then for $k=2$ simply observe,
$
S(g_1)S(g_2) = S(g_1 g_2).
$
\end{fact}
\bigskip

\begin{fact}\label{theta}
Define $\Theta S(\bar{g})= \sum_{\theta \in \Theta} S(\theta(\bar{g}))$ for $\bar{g}\in G^{k-1}.$ Then as stated in \cite{LaSu}
$\{\Theta S(\bar{g}):[\bar{g}]\in G^{k-1}/\Theta\}$ is an orthogonal basis for $P^{\Theta}_k$.
A simple calculation shows the following:
\begin{align*}
 & \Theta S(g_1,g_2,\cdots, g_{k-1}) \Theta S(h_1,h_2,\cdots, h_{k-1}) \\ 
 & \qquad = {(\sqrt {\lvert G \rvert})}^{(\lceil k/2 \rceil -1)}\sum_{{\theta}^{\dprime}\in \Theta}(\prod_{i=2}^{\lceil k/2\rceil}
 \delta(h_1 {\theta}^{\dprime}(g_{k+1-i}),h_i))~~~\times \\
 & \qquad \qquad \qquad \qquad \qquad \Theta S(h_1{\theta}^{\dprime}(g_1), h_1 {\theta}^{\dprime}(g_2), \cdots, h_1{\theta}^{\dprime}(g_{\lceil k/2\rceil}),
 h_{\lceil k/2 \rceil +1}, h_{\lceil k/2 \rceil +2},\cdots, h_{k-1})
\end{align*}
For $k=2$ the above is being interpreted as 
\begin{equation*}
 \Theta S(g) \Theta S(h) = \sum_{{\theta}^{\dprime} \in \Theta} \Theta S(g {\theta}^{\dprime}(h)).
\end{equation*}

\end{fact}
\bigskip
\begin{remark}
 Note that there is a slight correction in constant in Fact \ref{s} and \ref{theta} as compared to \cite{LaSu} Remark 3.3.1 (f) and (g) respectively.
\end{remark}

\begin{fact}\label{f}
Let $q$ be the biprojection corresponding to the intermediate subfactor $R^{\Theta}$ such that\\
$R^{G\rtimes \Theta}\subset R^{\Theta}\subset R$. In other words,
\begin{equation*}
 {\begin{minipage}{.1\textwidth}
    \centering
    \includegraphics[scale=.65]{q}
    \end{minipage}}= \frac{1}{\lvert \Theta \rvert}\sum_{\theta \in \Theta}
  {\begin{minipage}{.2\textwidth}
    \centering
    \includegraphics[scale=.6]{tt}
    \end{minipage}}
\end{equation*}

Then using exchange relation we easily get the following result as mentioned in \cite{LaSu}:
\begin{align*}
&  F_n(S((g_1,{\theta}_1), (g_2,{\theta}_2),\cdots,(g_{n-1},{\theta}_{n-1})))\\
& \qquad = \frac{1}{{\lvert \Theta\rvert}^n}\sum_{\substack{{\theta}\in \Theta\\ \bar{\gamma}\in {\Theta}^{n-1}}} S((\theta(g_1),\gamma_1),(\theta(g_2),\gamma_2),\cdots,(\theta(g_{n-1}),\gamma_{n-1}))
 \end{align*}
\end{fact}
\bigskip

\begin{remark}
 Observe that the formula in Fact \ref{f} depends only on the orbit of $(g_1,g_2,\cdots,g_{n-1})$ under $\Theta$. Following \cite{LaSu}
 we put,$$U(g_1,g_2,\cdots, g_{k-1})= \sum_{\substack{{\theta}\in \Theta\\ \bar{\gamma}\in {\Theta}^{k-1}}} S((\theta(g_1),\gamma_1),(\theta(g_2),\gamma_2),\cdots,(\theta(g_{k-1}),\gamma_{k-1})).$$
  Then it is simple to verify that $\{U(\bar{g}): [\bar{g}]\in G^{k-1}/\Theta\}$ is an
 orthogonal basis for $F_k(P^{(R^{G\rtimes \Theta}\subset R)}).$
\end{remark}
\bigskip

\begin{lemma}\label{u}
 \begin{align*}
 & U(g_1,g_2,\cdots,g_{k-1}) U(h_1,h_2,\cdots,h_{k-1})\\
 & \qquad = {(\sqrt {\lvert G\rvert})}^{(\lceil k/2 \rceil -1)}{(\sqrt {\lvert \Theta \rvert})}^{(\lceil k/2 \rceil -1)}(\lvert \Theta \rvert)^{\lfloor k/2 \rfloor}\sum_{{\theta}^{\dprime}\in \Theta}(\prod_{i=2}^{\lceil k/2\rceil}
 \delta(h_1 {\theta}^{\dprime}(g_{k+1-i}),h_i)) ~~~ \times\\
 & \qquad \qquad \qquad \qquad \qquad  U(h_1{\theta}^{\dprime}(g_1), h_1 {\theta}^{\dprime}(g_2), \cdots, h_1{\theta}^{\dprime}(g_{\lceil k/2\rceil}),
 h_{\lceil k/2 \rceil +1}, h_{\lceil k/2 \rceil +2},\cdots, h_{k-1})
\end{align*}
For $k=2$ the above is being interpreted as 
\begin{equation*}
 U(g) U(h)= \lvert \Theta \rvert \sum_{{\theta}^{\dprime}\in \Theta} \Theta S(g {\theta}^{\dprime}(h)).
\end{equation*}

\end{lemma}
\begin{proof}
 \begin{align*}
  & U(g_1,g_2,\cdots, g_{k-1}) U(h_1,h_2,\cdots, h_{k-1})\\
  & \qquad = (\sum_{\substack{{\theta}\in \Theta\\ \bar{\gamma}\in {\Theta }^{k-1}}} S((\theta(g_1),\gamma_1),(\theta(g_2),\gamma_2), \cdots, (\theta(g_{k-1}),\gamma_{k-1})))~~~ \times\\
  & \qquad  \qquad \qquad (\sum_{\substack{{\phi}\in \Theta\\ \bar{\sigma}\in {\Theta}^{k-1}}} S((\phi(h_1),\sigma_1), (\phi(h_2),\sigma_2), \cdots, (\phi(h_{k-1}),\sigma_{k-1})))\\
  & \qquad = (\sqrt{G})^{(\lceil k/2 \rceil -1)} (\sqrt{\lvert \Theta \rvert})^{(\lceil k/2 \rceil -1)} \sum_{\theta,\phi, \bar{\gamma}, \bar{\sigma}}(\prod_{i=2}^{\lceil k/2 \rceil} \delta((\phi(h_1)\sigma_1 \theta(g_{k+1-i}),\sigma_1 \gamma_{k+1-i}), (\phi(h_i),\sigma_i)))~~\times\\
  & \qquad \qquad \qquad S((\phi(h_1)\sigma_1 \theta(g_1), \sigma_1 \gamma_1), (\phi(h_1)\sigma_1 \theta(g_2), \sigma_1 \gamma_2), \cdots,
  (\phi(h_1) \sigma_1 \theta(g_{\lceil k/2 \rceil}), \sigma_1 \gamma_{\lceil k/2 \rceil}), \\ 
  & \qquad \qquad \qquad(\phi(h_{\lceil k/2 \rceil +1}),\sigma_{\lceil k/2 \rceil +1}),\cdots,
  (\phi(h_{k-1}),\sigma_{k-1})))~~~~~~~~~~~\textrm{[Using~~Fact}~~~\ref{s}]\\ 
   & \qquad = (\sqrt{G})^{(\lceil k/2 \rceil -1)} (\sqrt{\lvert \Theta \rvert})^{(\lceil k/2 \rceil -1)} \sum_{\phi,\theta} \sum_{\substack{\sigma_1\in \Theta\\ \bar{\delta}\in {\Theta}^{k-1}\\ {\bar{\gamma}}_{(\lceil k/2 \rceil,k-1]}}}(\prod_{i=2}^{\lceil k/2 \rceil} \delta(\phi(h_1 {\phi}^{-1} \sigma_1 \theta(g_{k+1-i})),\phi(h_i))~~\times\\
    & \qquad \qquad \qquad S((\phi(h_1{\phi}^{-1}\sigma_1 \theta(g_1)), \delta_1), (\phi(h_1 {\phi}^{-1}\sigma_1 \theta(g_2)), \delta_2), \cdots,
  (\phi(h_1 {\phi}^{-1} \sigma_1 \theta(g_{\lceil k/2 \rceil})), \delta_{\lceil k/2 \rceil}), \\ 
  & \qquad \qquad \qquad (\phi(h_{\lceil k/2 \rceil +1}),\delta_{\lceil k/2 \rceil +1}),\cdots,
  (\phi(h_{k-1}),\delta_{k-1})))\\ 
   & \qquad = (\sqrt{G})^{(\lceil k/2 \rceil -1)} (\sqrt{\lvert \Theta \rvert})^{(\lceil k/2 \rceil -1)} (\lvert \Theta \rvert)^{(\lfloor k/2 \rceil-1)} \lvert \Theta \rvert \sum_{\phi,\bar{\delta}}\sum_{{\theta}^{\dprime}\in \Theta}(\prod_{i=2}^{\lceil k/2 \rceil} \delta(h_1 \theta^{\dprime}(g_{k+1-i}), h_i)~~\times\\
    & \qquad \qquad \qquad S((\phi(h_1\theta^{\dprime}(g_1)), \delta_1), (\phi(h_1\theta^{\dprime}(g_2)), \delta_2), \cdots,
  (\phi(h_1 \theta^{\dprime}(g_{\lceil k/2 \rceil})), \delta_{\lceil k/2 \rceil}), \\ 
  & \qquad \qquad \qquad (\phi(h_{\lceil k/2 \rceil +1}),\delta_{\lceil k/2 \rceil +1}),\cdots,
  (\phi(h_{k-1}),\delta_{k-1})))~~~~~~~~~\textrm{[Putting}~~{\phi}^{-1}\sigma_1 \theta = {\theta}^{\dprime}]
  \end{align*}
  This completes the proof.
\end{proof}

\begin{definition}\label{defi}
 Define linear maps ${\Phi}_k : P^{\Theta}_k\longmapsto F_k(P_k(G\ltimes \Theta))$ by,\par
$${\Phi}_k(\Theta (S(\bar{g})))= {\lvert \Theta\rvert}^{-\lfloor k/2 \rfloor}(\sqrt{\lvert \Theta\rvert})^{(1-\lceil k/2\rceil)} U(\bar{g}),$$ which also equals to
$${(\alpha(S))}^{1-k-\lfloor k/2 \rfloor} U(\bar{g}).$$ Here $S$ is the tangle as in Notation \ref{basis}, but unlabelled. Here $[\bar{g}]\in G^{k-1}/\Theta.$
\end{definition}

To prove Theorem \ref{subgroup} we need to check that the following equation holds:
\begin{equation} \label{sanat}
{\Phi}_{k_0}(Z_T(x_1\otimes\cdots \otimes x_b))=Z^{\prime}_T({\Phi}_{k_1}(x_1)\otimes \cdots \otimes {\Phi}_{k_b}(x_b))\}
\end{equation}
 for any tangle $T(= T^{k_0}_{k_1,\cdots,k_b}).$
 
In view of \cite{KodSun}[Theorem 3.3], it suffices to prove 
\begin{theorem}\label{generating tangles}
The collection $\mathscr{T}$ of those
tangles $T$ which satisfy Equation (\ref{sanat}) contains a class of `generating tangles' namely
$\mathscr{T} \supset \{1^{0_{+}},1^{0_{-}}\}\cup\{{\mathcal{E}}^k:k\geq 2\}\cup \{{(E^{\prime})}^k_k: k\geq 1\}\cup \{E^k_{k+1},M_k, I^{k+1}_k : k\in Col\} $
\end{theorem}

We prove in detail that $\mathscr{T}$ contains  the multiplication tangles and the right conditional expectation tangles. In other cases we just  sketch the proofs.
\begin{lemma}
 $M_k \in \mathscr{T}$.
\end{lemma}
\begin{proof}
 Firstly note,
 \begin{align*}
 & {\Phi}_k(Z_{M_k}(\Theta S(g_1,g_2,\cdots,g_{k-1})\otimes \Theta S(h_1,h_2,\cdots,h_{k-1})))\\
 & \qquad \qquad \qquad= {\Phi}_k(\Theta S(g_1,g_2,\cdots,g_{k-1}) \Theta S(h_1,h_2,\cdots,h_{k-1}))\\
 & \qquad \qquad \qquad= {(\sqrt {\lvert G\rvert})}^{(\lceil k/2 \rceil -1)}\sum_{{\theta}^{\dprime}\in \Theta}(\prod_{i=2}^{\lceil k/2\rceil}
 \delta(h_1 {\theta}^{\dprime}(g_{k+1-i}),h_i))\times \\
 & \qquad \qquad \qquad \qquad \qquad {\Phi}_k(\Theta S(h_1{\theta}^{\dprime}(g_1), h_1 {\theta}^{\dprime}(g_2), \cdots, h_1{\theta}^{\dprime}(g_{\lceil k/2\rceil}),
 h_{\lceil k/2 \rceil +1}, h_{\lceil k/2 \rceil +2},\cdots, h_{k-1}))\\
 &\qquad \qquad \qquad \qquad \qquad \qquad\qquad\qquad \qquad \qquad \textrm{[Using~~Fact}~~\ref{theta}]\\
 & \qquad \qquad \qquad= (\sqrt{\lvert \Theta\rvert})^{1-\lceil k/2 \rceil}(\lvert \Theta\rvert)^{-\lfloor k/2\rfloor}{(\sqrt {\lvert G\rvert})}^{(\lceil k/2 \rceil -1)}\sum_{{\theta}^{\dprime}}(\prod_{i=2}^{\lceil k/2\rceil}
 \delta(h_1 {\theta}^{\dprime}(g_{k+1-i}),h_i))\times \\
 & \qquad \qquad \qquad \qquad \qquad U(h_1{\theta}^{\dprime}(g_1), h_1 {\theta}^{\dprime}(g_2), \cdots, h_1{\theta}^{\dprime}(g_{\lceil k/2\rceil}),
 h_{\lceil k/2 \rceil +1}, h_{\lceil k/2 \rceil +2},\cdots, h_{k-1})\\
 &\qquad \qquad \qquad \qquad \qquad \qquad\qquad\qquad \qquad \qquad \textrm{[Definition}~~\ref{defi}]
 \end{align*} 
 On the other hand, since $\alpha(M_k) =1$ the following equations hold:
 \begin{align*}
  & Z^{\prime}_{M_k}({\Phi}_k(\Theta S(g_1,g_2,\cdots,g_{k-1}))\otimes {\Phi}_k(\Theta S(h_1,h_2,\cdots, h_{k-1})))\\
  &  = \alpha(M_k) F_k(Z_{M_k}((\lvert \Theta\rvert)^{-2\lfloor  k/2 \rfloor} (\sqrt{\lvert\Theta \rvert})^{2(1-\lceil k/2 \rceil)}U(g_1,g_2,\cdots,g_{k-1})\otimes 
   U(h_1,h_2,\cdots,h_{k-1})))\\
   &\qquad \qquad \qquad \qquad \qquad \qquad\qquad\qquad \qquad \qquad \qquad \qquad \qquad \textrm{[Definition}~~\ref{defi}]\\
  & = {(\sqrt {\lvert G\rvert})}^{(\lceil k/2 \rceil -1)}{(\sqrt {\lvert\Theta\rvert})}^{(\lceil k/2 \rceil -1)}(\lvert\Theta\rvert)^{\lfloor k/2 \rfloor}(\lvert\Theta\rvert)^{-2\lfloor  k/2 \rfloor} (\sqrt{\lvert\Theta\rvert})^{2(1-\lceil k/2 \rceil)} 
   \sum_{{\theta}^{\dprime}\in \Theta}(\prod_{i=2}^{\lceil k/2\rceil}
 \delta(h_1 {\theta}^{\dprime}(g_{k+1-i}),h_i)) \\
  & \qquad \qquad \qquad \qquad U(h_1{\theta}^{\dprime}(g_1), h_1 {\theta}^{\dprime}(g_2), \cdots, h_1{\theta}^{\dprime}(g_{\lceil k/2\rceil}),
 h_{\lceil k/2 \rceil +1}, h_{\lceil k/2 \rceil +2},\cdots, h_{k-1}\\
 &\qquad \qquad \qquad \qquad \qquad \qquad\qquad\qquad \qquad \qquad \qquad \qquad \qquad \textrm{[by~~Lemma~~\ref{u}]}\\
 & = {(\sqrt {\lvert G\rvert})}^{(\lceil k/2 \rceil -1)}(\lvert\Theta\rvert)^{-\lfloor  k/2 \rfloor} (\sqrt{\lvert\Theta\rvert})^{1-\lceil k/2 \rceil)}\sum_{{\theta}^{\dprime}}(\prod_{i=2}^{\lceil k/2\rceil}
 \delta(h_1 {\theta}^{\dprime}(g_{k+1-i}),h_i))\\ 
 & \qquad \qquad \qquad \qquad U(h_1{\theta}^{\dprime}(g_1), h_1 {\theta}^{\dprime}(g_2), \cdots, h_1{\theta}^{\dprime}(g_{\lceil k/2\rceil}),
 h_{\lceil k/2 \rceil +1}, h_{\lceil k/2 \rceil +2},\cdots, h_{k-1})
 \end{align*}
 This completes the proof.
\end{proof}
\begin{lemma}
 $E^k_{k+1} \in \mathscr{T}$.
\end{lemma}
\begin{proof}
 \textbf{Case I}: $k=2n$. Put $T= E^k_{k+1}$.
 \bigskip

 For $n=1$, use relation $1$ to get
 $$Z_T(S(g_1,g_2))= S({g_1}^{-1}).$$

 If $n\geq 2$, we  again using relation $1$  get the following result easily:
 \begin{equation}\label{a}
 Z_T(S(g_1,g_2,\cdots, g_k))= S(g_1,g_2,\cdots,g_{(k/2)}, g_{(\frac{k}{2}+2)},\cdots,g_k) 
 \end{equation}
 Also observe, $\alpha(T)= {\lvert \Theta \rvert}^{-1/2}$.
 
 We show for $n \geq 2, T\in \mathscr{T} (n=1$ is exactly similar).
 \begin{align*}
 & {\Phi}_{k}(Z_T(\Theta S(g_1,g_2,\cdots,g_k)))\\
 &\qquad ={\Phi}_{k}(Z_T(\sum_{\theta \in \Theta}S(\theta(g_1),\theta(g_2),\cdots, \theta(g_k))))\\
 & \qquad= {\Phi}_{k}(\sum_{\theta}Z_T(S(\theta(g_1),\theta(g_2),\cdots, \theta(g_k))))\\
 & \qquad= {\Phi}_{k}(\sum_{\theta}S(\theta(g_1),\theta(g_2),\cdots,\theta(g_{k/2}),\theta(g_{(\frac{k}{2}+2)}),\cdots,\theta(g_k)))~~~\textrm{[by~~Equation~~\ref{a}]}\\
 &\qquad= {\Phi}_{k}(\Theta S(g_1, g_2,\cdots,g_{(k/2)},g_{(\frac{k}{2}+2)},\cdots,g_k))\\
 & \qquad= (\sqrt{\lvert \Theta \rvert})^{(1-\lceil k/2 \rceil)}(\lvert \Theta \rvert)^{-\lfloor k/2 \rfloor}U(g_1,g_2,\cdots,g_{(k/2)},g_{(\frac{k}{2}+2)},\cdots,g_k).
 \end{align*}
On the other hand,
\begin{align*}
 & Z^{\prime}_T({\Phi}_{k+1}(\Theta S(g_1,g_2,\cdots,g_k)))\\
 & \qquad= \alpha(T)F_k(Z_T((\sqrt{\lvert \Theta \rvert})^{(1-\lceil \frac{k+1}{2}\rceil)}(\lvert \Theta \rvert)^{-\lfloor\frac{k+1}{2}\rfloor}U(g_1,g_2,\cdots,g_k)))\\
 &\qquad= (\lvert \Theta \rvert)^{-1/2}(\sqrt{\lvert \Theta \rvert})^{(1-\lceil \frac{k+1}{2}\rceil)}(\lvert \Theta \rvert)^{-\lfloor \frac{k+1}{2}\rfloor}F_k(Z_T(\sum_{\substack{\theta \in \Theta\\ \bar{\gamma}\in {\Theta}^k}}(\theta(g_1),\gamma_1),(\theta(g_2),\gamma_2),\cdots,(\theta(g_k),\gamma_k))))\\
 &\qquad= (\lvert \Theta \rvert)^{-1/2}(\sqrt{\lvert \Theta \rvert})^{(1-\lceil \frac{k+1}{2}\rceil)}(\lvert \Theta \rvert)^{-\lfloor \frac{k+1}{2}\rfloor}
 F_k(\sum_{\theta,\bar{\gamma}}(\theta(g_1),\gamma_1),(\theta(g_2),\gamma_2),\cdots\\
 &\qquad \qquad \qquad \qquad \qquad\qquad\qquad \qquad \cdots,(\theta(g_{(k/2)}),\gamma_{(k/2)}),(\theta(g_{(\frac{k}{2}+2)}),\gamma_{(\frac{k}{2}+2)}),\cdots,(\theta(g_k),\gamma_k)))\\
 &\qquad= (\lvert \Theta \rvert)^{-1/2}(\sqrt{\lvert \Theta \rvert})^{(1-\lceil \frac{k+1}{2} \rceil)}(\lvert \Theta \rvert)^{-\lfloor \frac{k+1}{2} \rfloor} \lvert \Theta \rvert
 F_k(U(g_1,g_2,\cdots,g_{(k/2)},g_{(\frac{k}{2}+2)},\cdots,g_k)).\\
 & \qquad= (\sqrt{\lvert \Theta \rvert})^{(2-\lceil \frac{k+1}{2} \rceil)} (\lvert \Theta \rvert)^{-\lfloor \frac{k+1}{2} \rfloor} U(g_1,g_2,\cdots,g_{(k/2)},g_{(\frac{k}{2}+2)},\cdots,g_k)
 \end{align*}
Simple algebraic calculation tells us,
$\lceil (k+1)/2\rceil= \lceil (k/2)\rceil +1,$ and
$ \lfloor (k+1)/2 \rfloor= \lfloor k/2 \rfloor.$
Thus we have proved,
$$
 {\Phi}_{k}(Z_T(\Theta S(g_1,g_2,\cdots,g_k))) = Z^{\prime}_T({\Phi}_{k+1}(\Theta S(g_1,g_2,\cdots,g_k)))
$$
In other words, $E^k_{k+1} \in \mathscr{T}$.
\bigskip

\textbf{Case II}: $k=2n-1.$ Put $T= E^k_{k+1}$.
\bigskip

The case $n=1$ is trivial.
\bigskip

For $n\geq 2$ using relation $2$ as in \cite{LaSu} and exchange relation we easily get:
\begin{equation}\label{b}
Z_T(S(g_1,g_2,\cdots,g_k))= \sqrt{\lvert G\rvert} \delta(g_{(\lceil k/2 \rceil)},g_{(\lceil k/2 \rceil+1)}) S(g_1,g_2,\cdots, g_{(\lceil k/2\rceil)}, g_{(\lceil k/2 \rceil +2)},\cdots, g_k).
\end{equation}
 Then the following equations are easy to check:
\begin{align*}
 &{\Phi}_k(Z_T(\Theta S(g_1,g_2,\cdots,g_k)))\\
 &\qquad= {\Phi}_k(Z_T(\sum_{\theta\in \Theta} S(\theta(g_1),\theta(g_2),\cdots,\theta(g_k))))\\
 &\qquad= {\Phi}_k(\sum_{\theta}\sqrt{\lvert G \rvert}~~~\delta(g_{(\lceil k/2 \rceil)},g_{(\lceil k/2 \rceil+1)})S(\theta(g_1),\theta(g_2),\cdots,\theta(g_{(\lceil k/2\rceil)}), \theta(g_{(\lceil k/2 \rceil +2)}),\cdots,\theta(g_k))\\
 &\qquad \qquad \qquad \qquad \qquad \qquad \qquad \textrm{[by~~Equation~~\ref{b}]}\\
 & \qquad= {\Phi}_k(\sqrt{\lvert G\rvert}~~~\delta(g_{(\lceil k/2 \rceil)},g_{(\lceil k/2 \rceil+1)}) \Theta S(g_1,g_2,\cdots, g_{(\lceil k/2\rceil)}, g_{(\lceil k/2 \rceil +2)},\cdots, g_k))\\
 & \qquad= \sqrt{\lvert G\rvert}~~~\delta(g_{(\lceil k/2 \rceil)},g_{(\lceil k/2 \rceil+1)}) (\sqrt{\lvert \Theta \rvert})^{(1-\lceil k/2 \rceil)}(\lvert \Theta \rvert)^{-\lfloor k/2 \rfloor}
 U(g_1,g_2,\cdots, g_{(\lceil k/2\rceil)}, g_{(\lceil k/2 \rceil +2)},\cdots, g_k).
 \end{align*}
On the other hand,
\begin{align*}
 & Z^{\prime}_T({\Phi}_{k+1}(\Theta S(g_1,g_2,\cdots,g_k)))\\
 &\qquad = (\lvert \Theta \rvert)^{1/2} F_k(Z_T((\sqrt{\lvert \Theta \rvert})^{(1-\lceil \frac{k+1}{2}\rceil)}(\lvert \Theta \rvert)^{- \lfloor\frac{k+1}{2}\rfloor} U(g_1,g_2,\cdots,g_k)))~~\textrm{[since}~~\alpha(T) = {\lvert \Theta \rvert}^{1/2}].\\
 & \qquad = (\lvert \Theta \rvert)^{1/2}(\sqrt{\lvert \Theta \rvert})^{(1-\lceil \frac{k+1}{2}\rceil)}(\lvert \Theta \rvert)^{- \lfloor\frac{k+1}{2}\rfloor}\\
 &\qquad\qquad \qquad \qquad   F_k(Z_T(\sum_{\substack{\theta\in \Theta\\ \bar{\gamma}\in {\Theta}^k}} S((\theta(g_1),\gamma_1),(\theta(g_2),\gamma_2),\cdots,(\theta(g_k),\gamma_k))))\\
 &\qquad=  (\lvert \Theta \rvert)^{1/2}(\sqrt{\lvert \Theta \rvert})^{(1-\lceil \frac{k+1}{2}\rceil)}(\lvert \Theta \rvert)^{- \lfloor\frac{k+1}{2}\rfloor}
 (\sqrt{\lvert G\rvert} \sqrt{\lvert \Theta \rvert})\\
 & \qquad \qquad \qquad \qquad \delta(g_{(\lceil k/2 \rceil)},g_{(\lceil k/2 \rceil+1)}) F_k(U(g_1,g_2,\cdots, g_{(\lceil k/2\rceil)}, g_{(\lceil k/2 \rceil +2)},\cdots, g_k))~~\textrm{[by~~Equation~~\ref{b}]}\\
 &\qquad= \sqrt{\lvert G\rvert}(\sqrt{\lvert \Theta \rvert})^{(1-\lceil \frac{k+1}{2}\rceil)}(\lvert \Theta \rvert)^{(1- \lfloor\frac{k+1}{2}\rfloor)}
 \delta(g_{(\lceil k/2 \rceil)},g_{(\lceil k/2 \rceil+1)}) U(g_1,g_2,\cdots, g_{(\lceil k/2\rceil)}, g_{(\lceil k/2 \rceil +2)},\cdots, g_k)
 \end{align*}
In this case observe that,
$
\lceil k/2 \rceil= \lceil (k+1)/2 \rceil
$ and
$
\lfloor (k+1)/2 \rfloor = \lfloor k/2 \rfloor +1.
$
Thus,
$$
 {\Phi}_{k}(Z_T(\Theta S(g_1,g_2,\cdots,g_k))) = Z^{\prime}_T({\Phi}_{k+1}(\Theta S(g_1,g_2,\cdots,g_k)))
$$
In other words, $E^k_{k+1} \in \mathscr{T}$.
\end{proof}
 \begin{lemma}
  $I^{k+1}_k \in \mathscr{T}$.
 \end{lemma}
\begin{proof}
 Put $T= I^{k+1}_k.$  Clearly, $\alpha(T)=1.$
 
 Put
\begin{equation*}
 h= \frac{1}{\sqrt{\lvert G \rvert}}\sum_{g\in G}
 {\begin{minipage}{.2\textwidth}
    \centering
    \includegraphics[scale=.7]{gg}
    \end{minipage}}
\end{equation*}
 Clearly $\theta(h)=h$ for all $\theta \in \Theta.$
\bigskip

 It suffices to check that:
 \begin{equation*}
 Z_T(S(g_1,g_2,\cdots,g_{k-1}))= S(g_1,g_2,\cdots,g_{(k/2)},h,g_{(k/2+1)},\cdots, g_{k-1}).
 \end{equation*}
 for $k=2n.$

and

 \begin{equation*}
 Z_T(S(g_1,g_2,\cdots,g_{k-1})) = S(g_1, g_2,\cdots, g_{(\lceil k/2 \rceil)}, g_{(\lceil k/2 \rceil)},g_{(\lceil k/2 \rceil+1)},\cdots, g_{k-1}).
 \end{equation*}
 for $k=2n-1.$

The proof of the above two equations is routine, and omitted.
 \end{proof}
 
\begin{lemma}
 ${(E^{\prime})}^k_k\in \mathscr{T}.$
\end{lemma}
\begin{proof}
We omit the details.
Put $T= {(E^{\prime})}^k_k.$
In this case $\alpha(T)= (\lvert \Theta \rvert)^{1/2}$.
Then  simply observe:
 \begin{equation*}
 Z_T (S(g_{1},g_2,\cdots,g_{k-1}))= \sqrt{\lvert G \rvert} \delta(g_{1},e) S(e, g_2,\cdots,g_{k-1}).
 \end{equation*}
 \end{proof}
\bigskip

This complets the proof of Theorem \ref{generating tangles}.
\bigskip

Lastly we apply Theorem \ref{main1} to conclude that the proof of Theorem \ref{subgroup} is now complete.
\section{Appendix}

As promised in the introduction, we here describe the tower of iterated basic construction of  $N\subseteq Q$ in terms of the the corresponding tower of $N \subseteq M$. 
Firstly we state  the following lemma which we will use to prove this.
\begin{lemma}\cite{Bak}
\label{fvrt}
Let $N\subseteq M$ be an inclusion of Type  $II_1$ factors.
Assume $\{\lambda_i:i\in{1,2,..n}\}$
is a basis for $M/N$ (in the sense used in \cite{JoSu}). Let $P$ be a $II_1$ factor such that $P$ contains $M$  and also  contains a projection $f$ such that
$\sum_{i=1}^n{{\lambda_i}^*f{\lambda_i}}=1$ and satisfies
further the following two properties :\par$1)fxf=E_N(x)f$  for all $x\in
M$ and\par2)\{${\tau}^{-1/2}f{\lambda_i}$\} is a basis for $P/M $.
 \\Then there exists an isomorphism from $M_1=\langle M,e_1\rangle$ onto $P$ which maps
$e_1$ to f.
\end{lemma}
The following  well known fact is often useful:
\begin{fact}\label{f:localtraceformula}
Given a $II_1$ factor $A$ and projections $r \in A$ and $s \in
(rAr)'$, we have
\begin{equation} \label{eq1}  tr_{rArs}(rzrs)= (tr_A(r))^{-1} tr_A (rzr)
\end{equation}
for all {$z \in A$.}
\end{fact}
\begin{theorem} \label{T:intsi}
Given $N\ss Q  \ss M$ and the notation introduced above, set $\p =e_{0,1}\ e_{0,3} \ e_{0,5} \cdots e_{0, 2n-1}$.  Then the chain
\[ N\p  \ss \p M\p  \ss \p M_1\p  \ss \p M_2\p  \ss \dots \ss \p M_{2n-1}\p  \]
is isomorphic to the first $2n-1$ steps of the basic construction of
$N \ss Q $.  The Jones projections are given by $e_{0,2i}\p : L^2( \p
M_{2i-1}\p ) \rightarrow L^2( \p M_{2i-2}\p )$ and $e_{1,2i+1}\p :
L^2( \p M_{2i}\p ) \rightarrow L^2( \p M_{2i-1}\p )$.  The unique
normalized trace on the chain, denoted $tr_{N \ss Q }$, is given by
$tr_{N \ss Q }(x) = [M:Q ]^{n}tr_{N \ss M}(x) $.
\end{theorem}
\begin{proof}
We put, $p_{[0,2n-1]}= e_{0,1}\ e_{0,3} \ e_{0,5} \cdots e_{0, 2n-1}$. The final trace assertion is immediate from the fact $tr_{N \ss M}(\p_{[0,2n-1]}) = [M:Q ]^{-n}$.
 It suffices to show that the above chain is a basic
construction and that the inclusion $N\p_{[0,2n-1]}\ss \p_{[0,2n-1]}M \p_{[0,2n-1]}$ is isomorphic to $N \ss Q $.  
We do this in several steps.
\bigskip

$ \noindent \textbf{Step 1}:$ Firstly we show, $Ne_{0,1}\subseteq e_{0,1}Me_{0,1} \subseteq e_{0,1}M_1e_{0,1}$ is isomorphic to the (first step)  basic construction of $N\subseteq Q$, where the corresponding Jones' projection is given by $e_{1,1}e_{0,1}(= e_{1,1})$. We prove this using Lemma \ref{fvrt}.
\par Since $Q^{\prime}$ is a von Neumann algebra and $e_{0,1}\in Q^{\prime}, Qe_{0,1}$ is a von Neumann algebra. Again, as $e_{0,1} \in N^{\prime}, Ne_{0,1}$ is also a von Neumann algebra. Since $e_{0,1}Me_{0,1}= Qe_{0,1}$ it is clear that $(Ne_{0,1}\subseteq e_{0,1}Me_{0,1}) \cong (N \subseteq Q)$
via the map $ q \longmapsto q e_{0,1} $ for $q\in Q $. In particular, $[Qe_{0,1}:Ne_{0,1}]= [Q:N]$. Let $\{\lambda_i\}$ be a basis for $Q/N$, which always exists by \cite{PiPo}, then since  $E^{Qe_{0,1}}_{Ne_{0,1}}(qe_{0,1})= E^Q_N(q)e_{0,1}, \{\lambda_i e_{0,1}\}$ is a basis for $Qe_{0,1}/Ne_{0,1}$. 
Now, $e_{1,1} \in e_{0,1}M_1e_{0,1}$ since $e_{1,1}\leq e_{0,1}$ and observe,
\begin{align*}
 & \sum{{(\lambda_i e_{0,1})}^* e_{1,1} \lambda_i e_{0,1}}\\
 & \qquad = \sum{{\lambda_i}^* e_{0,1} e_{1,1} e_{0,1} \lambda_i}~~\textrm{~[since}~~ e_{0,1}\in Q^{\prime}]\\
 & \qquad = \sum{{\lambda_i}^* e_{1,1} \lambda_i}~~~\textrm{~[since}~~~e_{1,1}= e_{1,1}e_{0,1}= e_{0,1}e_{1,1}]\\
 & \qquad = \sum{{\lambda_i}^* e^Q_N e_{0,1}\lambda_i}~~~\textrm{~[since}~~~e^Q_N e_{0,1}= e_{1,1}]\\
 & \qquad = \sum{({\lambda_i}^* e^Q_N \lambda_i}) e_{0,1}\\
 & \qquad = e_{0,1} ~~~\textrm{~[since}~~\sum{{\lambda_i}^* e^Q_N \lambda_i}=1].
\end{align*}
Also, 
\begin{equation}\label{p1}
 e_{1,1} (e_{0,1}m e_{0,1}) e_{1,1}= e_{1,1} m e_{1,1}= E^M_N(m) e_{1,1}
\end{equation} and,
\begin{equation}\label{p2}
E^{Qe_{0,1}}_{Ne_{0,1}}(e_{0,1} m e_{0,1})e_{1,1}= E^{Qe_{0,1}}_{Ne_{0,1}}(E^M_Q(m)e_{0,1})e_{1,1}=E^Q_N(E^M_Q(m))e_{0,1} e_{1,1}= E^M_N(m) e_{1,1}
\end{equation}
  for all $m\in M.$
  Equation (\ref{p1}) and (\ref{p2}) implies $e_{1,1} (e_{0,1}m e_{0,1}) e_{1,1}= E^{Qe_{0,1}}_{Ne_{0,1}}(e_{0,1} m e_{0,1})e_{1,1}.$
  We next show, 
  \begin{equation}\label{p3}
  E^{e_{0,1}M_1e_{0,1}}_{Q e_{0,1}}(x)= [M:Q] E^{M_1}_Q(x)e_{0,1}.
  \end{equation}
  for all $x \in e_{0,1}M_1e_{0,1}.$
  This follows from the following array of equations $\forall q\in Q $:
  \begin{align*}
   & tr_{N\subseteq Q}([M:Q] E^{M_1}_Q (x)e_{0,1}.qe_{0,1})\\
   & \qquad = [M:Q]^2 tr(E^{M_1}_Q(x)qe_{0,1})\\
   &\qquad = [M:Q]^2 tr(E^{M_1}_Q(xq)e_{0,1})\\
   & \qquad= [M:Q] tr(E^{M_1}_Q(xq)) \textrm{~~~[since}~~~tr(e_{0,1})= {[M:Q]}^{-1}]\\
   & \qquad= [M:Q] tr(xq)\\
   & \qquad = tr_{N\subseteq Q}(e_{0,1}xq) \textrm{~~~[since}~~~e_{0,1}x=x]\\
   & \qquad= tr_{N\subseteq Q}(x.qe_{0,1}).
  \end{align*}
  Then we show, $\{\sqrt{[Q:N]}e_{1,1}\lambda_i e_{0,1}\} $ is a basis for $e_{0,1}M_1 e_{0,1}/e_{0,1} M e_{0,1}.$
  To prove this firstly note that,by Jones's local index formula (see \cite{Jo1} or sections 2.2-2.3 of \cite{JoSu}) and extremality (see \cite{Po} page 176) the following equation holds
  \begin{equation}\label{p4}
   [e_{0,1}M_1 e_{0,1}:Q e_{0,1}]= {tr(e_{0,1})}^2 [M_1:Q]= [Q:N]= [Qe_{0,1}:Ne_{0,1}]
   \end{equation}

  Then,the  following array of equations hold:
  \begin{align*}
   & E^{e_{0,1}M_1 e_{0,1}}_{Qe_{0,1}}[(\sqrt{[Q:N]} e_{1,1} \lambda_i e_{0,1})({\sqrt{[Q:N]} e_{1,1} \lambda_j e_{0,1})}^*]\\
   & \qquad = [Q:N] E^{e_{0,1}M_1 e_{0,1}}_{Qe_{0,1}}(e_{1,1}e_{0,1}\lambda_i {\lambda}^*_je_{1,1})        ~~\textrm{~[since}~~ e_{0,1}\in Q^{\prime}]\\
   & \qquad = [Q:N] E^{e_{0,1}M_1 e_{0,1}}_{Qe_{0,1}}(e_{0,1}e_{1,1}\lambda_i {\lambda}^*_je_{1,1}e_{0,1})~~~\textrm{~[since}~~~e_{1,1}= e_{1,1}e_{0,1}= e_{0,1}e_{1,1}]\\
   & \qquad = [Q:N] [M:Q] E^{M_1}_Q(e_{0,1}e_{1,1}\lambda_i {\lambda}^*_je_{1,1}e_{0,1})e_{0,1}~~~~\textrm{~~~[by~~Equation}~~(\ref{p3})]\\
   & \qquad= [M:N] E^{M_1}_Q(e_{1,1} \lambda_i{\lambda}^*_je_{1,1})e_{0,1}\\
   & \qquad = [M:N] E^{M_1}_Q[E^M_N(\lambda_i{\lambda}^*_j)e_{1,1}]e_{0,1}\\
   & \qquad = [M:N]E^Q_N(\lambda_i{\lambda}^*_j)E^{M_1}_Q(e_{1,1}) e_{0,1}\\
   & \qquad = [M:N] E^Q_N(\lambda_i{\lambda}^*_j)E^M_Q(E^{M_1}_M(e_{1,1})) e_{0,1}\\
   & \qquad = E^Q_N(\lambda_i{\lambda}^*_j)e_{0,1}\\
   & \qquad = E^{Qe_{0,1}}_{Ne_{0,1}}(\lambda_i{\lambda}^*_j e_{0,1})\\
   & \qquad = E^{Qe_{0,1}}_{Ne_{0,1}}(\lambda_i e_{0,1}.e_{0,1}{\lambda}^*_j)   ~~\textrm{~[since}~~ e_{0,1}\in Q^{\prime}]
  \end{align*}
  Now as $ \{\lambda_i e_{0,1}\}$ is a basis for $Qe_{0,1}/Ne_{0,1}$ the last equation in the above array of equations together with Equation (\ref{p4}) tells that
  $\{\sqrt{[Q:N]}e_{1,1}\lambda_i e_{0,1}\} $ is a basis for $e_{0,1}M_1 e_{0,1}/e_{0,1} M e_{0,1}.$ Here we have used Theorem 2.2 in \cite{Bak}. Now applying Lemma \ref{fvrt}
  we get the desired result.\vspace{4mm}\\
 $\noindent \textbf{Step 2}:$ Here again we have $N\subseteq Q \subseteq M$, with biprojection $e_{0,1}$. We claim $(p_{[0,3]}M p_{[0,3]}\subseteq p_{[0,3]}M_1 p_{[0,3]}\subseteq p_{[0,3]}M_2 p_{[0,3]})\cong(Q \subseteq Q_1 \subseteq Q_2)$.
 Here Jones' projection is given by $e_{0,2}p_{[0,3]}$.
 We will again apply Lemma \ref{fvrt}.\\
 Firstly, as $e_{0,3}$ commutes with $e_{0,1}$ and every element of $M_1$, it follows from above discussion  that \\(a)$\{\sqrt{[Q:N]} e_{1,1}\lambda_i p_{[0,3]}\}$ is 
 a basis for $p_{[0,3]} M_1 p_{[0,3]}/p_{[0,3]} Mp_{[0,3]}$.\vspace{2mm}\\
 Next we show that \\
 (b) for $m_1\in M_1,$
 \begin{equation}\label{p5}
 (e_{0,2} p_{[0,3]})(p_{[0,3]} m_1 p_{[0,3]})(e_{0,2}p_{[0,3]})= E^{p_{[0,3]} M_1 p_{[0,3]}}_{p_{[0,3]} M p_{[0,3]}}(p_{[0,3]} m_1 p_{[0,3]})e_{0,2}p_{[0,3]}.
 \end{equation}
  As $e_{0,3}$ commutes with $e_{0,1}$ and every element of $M_1$, it follows that,
 \begin{align*}
  & E^{p_{[0,3]} M_1 p_{[0,3]}}_{p_{[0,3]} M p_{[0,3]}}(p_{[0,3]} m_1 p_{[0,3]})\\
  & \qquad = E^{e_{0,1} M_1 e_{0,1} e_{0,3}}_ {Q e_{0,1}e_{0,3}}(e_{0,1} m_1 e_{0,1}e_{0,3})\\
  & \qquad = E^{e_{0,1} M_1 e_{0,1}}_{Q e_{0,1}}(e_{0,1} m_1 e_{0,1})e_{0,3}~~~\textrm{~~[by~~Fact~~\ref{f:localtraceformula}]}\\
  & \qquad = [M:Q] E^{M_1}_Q(e_{0,1}m_1 e_{0,1})p_{[0,3]}~~\textrm{~~[by~~~Equation}~~~(\ref{p3})]
 \end{align*}
 Since $e_{0,1}\in P_1$ it is easy to see that, $E^{M_1}_{P_1}(e_{0,1}m_1e_{0,1})=E^{M_1}_{P_1}(e_{0,1}m_1e_{0,1})e_{0,1}=me_{0,1}$ for some unique $m \in M$, which exists by \cite{PiPo}.\\
 That implies,
 \begin{align*}
  &E^{M_1}_M(E^{M_1}_{P_1}(e_{0,1}m_1e_{0,1}))\\
  & \qquad \qquad= m E^{M_1}_M(e_{0,1})\\
  & \qquad \qquad= m E^{P_1}_M(e_{0,1})\\
  & \qquad \qquad= m{[M:Q]^{-1}}
 \end{align*} That is,
$ E^{M_1}_M(e_{0,1}m_1e_{0,1})=m{[M:Q]}^{-1}.$ Therefore $ m=[M:Q] E^{M_1}_M(e_{0,1}m_1e_{0,1}).$ Thus,
\begin{equation}\label{pop}
 E^{M_1}_{P_1}(e_{0,1}m_1e_{0,1})= [M:Q] E^{M_1}_M(e_{0,1}m_1e_{0,1}) e_{0,1}
\end{equation}

Thus,
\begin{align*}
 &  (e_{0,2} p_{[0,3]})(p_{[0,3]} m_1 p_{[0,3]})(e_{0,2}p_{[0,3]})\\
 & \qquad = e_{0,2}p_{[0,3]}m_1p_{[0,3]}e_{0,2}p_{[0,3]}\\
 & \qquad = e_{0,3}e_{0,2}e_{0,1}m_1 e_{0,1}e_{0,2}e_{0,3}~~~\textrm{~~~[by~~~Fact~~\ref{f:erel}(3)]}\\
 & \qquad = e_{0,1}{E^{M_1}_{P_1}}(e_{0,1}m_1e_{0,1})e_{0,2}e_{0,3}~~~\textrm{~~~[by~~~definition~~of}~~e_{0,2}]\\
 & \qquad= [M:Q]e_{0,1}{E^{M_1}_M}(e_{0,1}m_1e_{0,1})e_{0,1}e_{0,2}e_{0,3}~~~~\textrm{~~~[by~~~Equation~~(\ref{pop})]}\\
 & \qquad= [M:Q]E^M_Q(E^{M_1}_M(e_{0,1}m_1e_{0,1}))e_{0,1}e_{0,2}e_{0,3}\\
 & \qquad= [M:Q]E^{M_1}_Q(e_{0,1}m_1e_{0,1})p_{[0,3]}e_{0,2}p_{[0,3]}~~~\textrm{~~~[by~~~Fact~~\ref{f:erel}(3)]}\\
 & \qquad= E^{p_{[0,3]}M_1p_{[0,3]}}_{p_{[0,3]}Mp_{[0,3]}}({p_{[0,3]}m_1p_{[0,3]}})e_{0,2}p_{[0,3]}
\end{align*}
This completes the proof of (b).\\
(c) We need to prove $ [p_{[0,3]}M_2p_{[0,3]}:p_{[0,3]}M_1p_{[0,3]}]=[Q:N]$\newline
For this note,
\begin{align*}
 & p_{[0,3]}M_2p_{[0,3]}\\
 & \qquad= e_{0,1}e_{0,3}M_2e_{0,3}e_{0,1}\\
 & \qquad= e_{0,1}P_2e_{0,3}e_{0,1}~~~\textrm{~~~[by~~~definition~~of}~~e_{0,3}]\\
  &\qquad= p_{[0,3]}P_2p_{[0,3]}~~~\textrm{~~~[since}~~e_{0,3}\in P^{\prime}_2]
\end{align*}
It is trivial to see that, $[P_2:M_1] = [M_1:P_1] = \frac{[M_1:M]}{[P_1:M]}= \frac{[M:N]}{[M:Q]}=[Q:N]$.
Thus,
\begin{align*}
  &[Q:N]=[P_2:M_1]\\
  & \qquad \qquad=[e_{0,1}P_2e_{0,1}:e_{0,1}M_1e_{0,1}] \textrm~~~{~~~~~~~~[~~~~~~~as}~~~~e_{0,1}\in M_1\subseteq P_2]\\
& \qquad \qquad=[p_{[0,3]}P_2p_{[0,3]}:p_{[0,3]}M_1p_{[0,3]}]\\
& \qquad \qquad= [p_{[0,3]}M_2p_{[0,3]}:p_{[0,3]}M_1p_{[0,3]}]
\end{align*}
This proves (c).\\
(d) The following array of equations hold:
\begin{align*}
 & \sum{(e_{1,1} \lambda_i p_{[0,3]})}^* e_{0,2}p_{[0,3]}(e_{1,1} \lambda_i p_{[0,3]})\\
 & \qquad= \sum p_{[0,3]}{{\lambda}^*_i} e_{1,1}e_{0,2}e_{1,1} \lambda_i p_{[0,3]}\\
 & \qquad= [Q:N]^{-1} \sum p_{[0,3]} {{\lambda}^*_i}e_{1,1} \lambda_i p_{[0,3]}~~~\textrm{~~~[Fact}~~\ref{f:erel} (6)]\\
 & \qquad= [Q:N]^{-1} \sum p_{[0,3]} {{\lambda}^*_i} e^Q_N e_{0,1} \lambda_i p_{[0,3]}\\
 & \qquad= [Q:N]^{-1} p_{[0,3]}
\end{align*}
The last equation follows from the fact that $\{\lambda_i\}$ is a basis for $Q/N$ and hence $\sum{{\lambda}^*_i e^Q_N {\lambda}_i} = 1$.\\
(e) Firstly it is easy to check that,
\begin{equation}\label{e1}
 E^{P_2}_{M_1}(e_{0,2})= [M_1:P_1]^{-1}= [Q:N]^{-1}.
\end{equation}
Next we show that,
\begin{equation}\label{e2}
 E^{M_1}_{P_1}(e_{1,1})= [Q:N]^{-1}e_{0,1}
\end{equation}
To see this, note that by \cite{PiPo} there exists unique $m_0\in M$ such that $E^{M_1}_{P_1}(e_{1,1})= E^{M_1}_{P_1}(e_{1,1})e_{0,1} = m_0 e_{0,1}.$
Thus, $E^{M_1}_M(e_{1,1})= m_0 E^{M_1}_M(e_{0,1})$. Hence, $m_0 = [Q:N]^{-1}.$
We have,
\begin{align*}
 & [p_{[0,3]}M_2p_{[0,3]}:p_{[0,3]}M_1p_{[0,3]}]\\
 & \qquad= [Q:N]~~~~~~~~~~~~~~~~~~~~~~~~~~~~~~~~~~~~~~~~~~\textrm~~~~~~{[by~~~~~~ ~~(c)]}\\
 & \qquad= [Q p_{[0,1]}:N p_{[0,1]}]\\
 & \qquad= [p_{[0,1]}M_1p_{[0,1]}:p_{[0,1]}Mp_{[0,1]}]~~~~~~~~~~~~~~~~~~~~~~~~~~~~\textrm~~~{[by~~~Equation~~(\ref{p4})]}\\
 & \qquad= [p_{[0,3]}M_1p_{[0,3]}:p_{[0,3]}Mp_{[0,3]}]
\end{align*}
By Fact \ref{f:localtraceformula} it is trivial to check that for all $m_2 \in M_2$:
\begin{equation} \label{la}
 E^{p_{[0,3]}M_2 p_{[0,3]}}_{p_{[0,3]}M_1 p_{[0,3]}}(p_{[0,3]}m_2 p_{[0,3]})=p_{[0,3]}E^{M_2}_{M_1}(m_2)p_{[0,3]}
\end{equation}

Then the following array of equation hold:
\begin{align*}
 & E^{p_{[0,3]}M_2 p_{[0,3]}}_{p_{[0,3]}M_1 p_{[0,3]}} [(e_{0,2}e_{1,1} \lambda_i p_{[0,3]}){(e_{0,2}e_{1,1} \lambda_j p_{[0,3]})}^*]\\
 & \qquad \qquad= E^{p_{[0,3]}M_2p_{[0,3]}}_{p_{[0,3]}M_1p_{[0,3]}} (p_{[0,3]}e_{0,2}e_{1,1} \lambda_i {{\lambda}^*_j} e_{1,1}e_{0,2}p_{[0,3]})\\
 & \qquad \qquad= p_{[0,3]}E^{M_2}_{M_1} (e_{0,2}e_{1,1} \lambda_i {{\lambda}^*_j} e_{1,1}e_{0,2})p_{[0,3]}~~~~\textrm{~~[by~~Equation~~(\ref{la})]}\\
 & \qquad\qquad= p_{[0,3]}E^{M_2}_{M_1} (E^{M_1}_{P_1}(e_{1,1} \lambda_i {{\lambda}^*_j} e_{1,1})e_{0,2})p_{[0,3]}\\
 & \qquad \qquad= p_{[0,3]}E^{M_1}_{P_1} (e_{1,1} \lambda_i {{\lambda}^*_j} e_{1,1})E^{M_2}_{M_1}(e_{0,2}) p_{[0,3]}\\
 & \qquad\qquad= p_{[0,3]} E^Q_N(\lambda_i {{\lambda}^*_j})E^{M_1}_{P_1}(e_{1,1})E^{P_2}_{M_1}(e_{0,2})p_{[0,3]}\\
 & \qquad \qquad= p_{[0,3]} E^Q_N (\lambda_i {{\lambda}^*_j}){[Q:N]}^{-1} E^{M_1}_{P_1}(e_{1,1})p_{[0,3]}~~~\textrm{~~~[by ~~Equation~~(\ref{e1})]}\\
 & \qquad \qquad= p_{[0,3]}E^Q_N (\lambda_i {{\lambda}^*_j}){[Q:N]}^{-1} {[Q:N]}^{-1}e_{0,1}p_{[0,3]}~~~~\textrm{~~~~[by~~~Equation~~(\ref{e2})]}\\
 & \qquad \qquad= {[Q:N]}^{-2} p_{[0,3]} E^Q_N (\lambda_i {{\lambda}^*_j})p_{[0,3]}\\
 & \qquad \qquad= {[Q:N]}^{-2} E^{Q p_{[0,3]}} _{Np_{[0,3]}}(\lambda_i p_{[0,3]} p_{[0,3]} {{\lambda}^*_j})
\end{align*}
Since, $\{\lambda_i p_{[0,3]}\}$ is a basis for $Qp_{[0,3]}/Np_{[0,3]}$ it follows from Theorem $2.2$ in \cite{Bak} that $\{[Q:N] e_{0,2}p_{[0,3]}e_{1,1}\lambda_i p_{0,3}\}$ that is,
$\{[Q:N] e_{0,2}e_{1,1} \lambda_i p_{[0,3]}\}$ is a basis for $p_{[0,3]} M_2 p_{[0,3]}/p_{[0,3]}M_1 p_{[0,3]}$.\\
Thus  combining (a),(b),(c),(d) and (e) and applying Lemma \ref{fvrt} we complete the proof of Step 2. \vspace{4mm}\\
$ \noindent \textbf{Step 3}:$ \\
 In general apply Step 1 to  the following subfactors for $2n-1\geq i\geq 3 $ and $i$ odd:\\
$p_{[0,2n-i]} M_{2n-i} p_{[0,2n-i]}\subseteq p_{[0,2n-i]} P_{2n-{i+1}} p_{[0,2n-i]} \subseteq p_{[0,2n-i]} M_{2n-{i+1}} p_{[0,2n-i]}.$\\
Here biprojection is given by $ p_{[0,2n-i]} e_{0,2n-{i+2}}$. We get
$p_{[0,2n-{i+2}]} M_{2n-i} p_{[0,2n-{i+2}]}\subseteq p_{[0,2n-{i+2}]}\\ M_{2n-{i+1}} p_{[0,2n-{i+2}]}\subseteq p_{[0,2n-{i+2}]} M_{2n-{i+2}} p_{[0,2n-{i+2}]}$
is a Jones' tower with corresponding Jones' projection is given by $e_{1,2n-i+2} p_{[0,2n-{i+2}]}$. But $ e_{0,2n-1} e_{0,2n-3}\cdots e_{0,2n-{i+4}} \in M^{\prime}_{2n-i}, M^{\prime}_{2n-{i+1}}$ and $ M^{\prime}_{2n-{i+2}}.$
So, using Fact \ref{f:erel} (3) we have 
$$p_{[0,2n-1]} M_{2n-i} p_{[0,2n-1]}\subseteq p_{[0,2n-1]} M_{2n-{i+1}} p_{[0,2n-1]} \subseteq p_{[0,2n-1]} M_{2n-{i+2}} p_{[0,2n-{i+2}]}$$
is a Jones' tower with the Jones' projection $e_{1,2n-i+2} p_{[0,2n-1]}$. \vspace{3mm}\\
Apply Step 2 to the following subfactors for $2n-1\geq j\geq 3 $ and $j$ odd:\\
$p_{[0,2n-j]} M_{2n-{j-2}} p_{[0,2n-j]}\subseteq p_{[0,2n-j]} P_{2n-{j-1}} p_{[0,2n-j]} \subseteq p_{[0,2n-j]} M_{2n-{j-1}} p_{[0,2n-j]}.$\\
Here biprojection is given by $ p_{[0,2n-j]} e_{0,2n-j}$. Then we get 
$p_{[0,2n-{j+2}]} M_{2n-{j-1}} p_{[0,2n-{j+2}]}\subseteq p_{[0,2n-{j+2}]}\\ M_{2n-j} p_{[0,2n-{j+2}]} \subseteq p_{[0,2n-{j+2}]} M_{2n-{j+1}} p_{[0,2n-{j+2}]}$ is a 
Jones' tower with the corresponding Jones' projection is given by $e_{0,2n-j+1} p_{[0,2n-j+2]}$.
Again note that $ e_{0,2n-1} e_{0,2n-3}\cdots e_{0,2n-{j+4}} \in M^{\prime}_{2n-j-1}, M^{\prime}_{2n-j}$ and $ M^{\prime}_{2n-j+1}$. Thus it follows from Fact \ref{f:erel} (3) that
$$p_{[0,2n-1]} M_{2n-{j-1}} p_{[0,2n-1]}\subseteq p_{[0,2n-1]} M_{2n-j} p_{[0,2n-1]} \subseteq p_{[0,2n-1]} M_{2n-{j+1}} p_{[0,2n-1]}$$ is a Jones's tower with the corresponding Jones' projection is given by $e_{0,2n-j+1} p_{[0,2n-1]}$.\vspace{3mm}\\
Lastly note that, from Step 1 and Fact \ref{f:erel} it follows that $$Np_{[0,2n-1]}\subseteq p_{[0,2n-1]}Mp_{[0,2n-1]} \subseteq p_{[0,2n-1]}M_1p_{[0,2n-1]}$$ is a Jones' tower with the corresponding Jones' projection $e_{1,1}p_{[0,2n-1]}$.\par
Combining all the facts mentioned above we prove the result as stated in the theorem.
\end{proof}

\section*{Acknowledgement} The author sincerely thanks V.S. Sunder and Vijay Kodiyalam for innumerable enlightening discussions and fruitful advice. He also wishes to thank Sohan Lal Saini and Sebastein Palcoux for various useful discussions.

\end{document}